\documentclass[12pt]{amsart}

\usepackage{amsmath}
\usepackage[curve]{xypic}
\usepackage{caption}
\usepackage{xspace}
\usepackage{cite}
\usepackage{enumitem}
\usepackage{color}
\usepackage{url}
\usepackage[usenames,dvipsnames]{xcolor}
\usepackage{graphicx}

\makeatletter 
\def\@cite#1#2{{\m@th\upshape\bfseries%
[{#1\if@tempswa{\m@th\upshape\mdseries, #2}\fi}]}}
\makeatother 

\theoremstyle{plain}
\newtheorem{thm}{Theorem}[section]
\newtheorem{cor}[thm]{Corollary}
\newtheorem{prop}[thm]{Proposition}
\newtheorem{lem}[thm]{Lemma}

\theoremstyle{definition}
\newtheorem{defn}[thm]{Definition}

\newtheorem{ex}[thm]{Example}

\theoremstyle{remark}
\newtheorem{rem}[thm]{Remark}

\numberwithin{equation}{subsection}
\renewcommand{\bold}[1]{\medskip \noindent {\bf #1 }\nopagebreak}

\newcommand{\nc}{\newcommand}
\newcommand{\rnc}{\renewcommand}

\newcommand{\bk}{{\mathbf{k}}}

\nc\bA{\mathbb{A}}
\nc\bB{\mathbb{B}}
\nc\bC{\mathbb{C}}
\nc\bD{\mathbb{D}}
\nc\bE{\mathbb{E}}
\nc\bF{\mathbb{F}}
\nc\bG{\mathbb{G}}
\nc\bH{\mathbb{H}}
\nc\bI{\mathbb{I}}
\nc{\bJ}{\mathbb{J}} 
\nc\bK{\mathbb{K}}
\nc\bL{\mathbb{L}}
\nc\bM{\mathbb{M}}
\nc\bN{\mathbb{N}}
\nc\bO{\mathbb{O}}
\nc\bP{\mathbb{P}}
\nc\bQ{\mathbb{Q}}
\nc\bR{\mathbb{R}}
\nc\bS{\mathbb{S}}
\nc\bT{\mathbb{T}}
\nc\bU{\mathbb{U}}
\nc\bV{\mathbb{V}}
\nc\bW{\mathbb{W}}
\nc\bY{\mathbb{Y}}
\nc\bX{\mathbb{X}}
\nc\bZ{\mathbb{Z}}
\nc\cA{\mathcal{A}}
\nc\cB{\mathcal{B}}
\nc\cC{\mathcal{C}}
\rnc\cD{\mathcal{D}}
\nc\cE{\mathcal{E}}
\nc\cF{\mathcal{F}}
\nc\cG{\mathcal{G}}
\rnc\cH{\mathcal{H}}
\nc\cI{\mathcal{I}}
\nc{\cJ}{\mathcal{J}} 
\nc\cK{\mathcal{K}}
\rnc\cL{\mathcal{L}}
\nc\cM{\mathcal{M}}
\nc\cN{\mathcal{N}}
\nc\cO{\mathcal{O}}
\nc\cP{\mathcal{P}}
\nc\cQ{\mathcal{Q}}
\rnc\cR{\mathcal{R}}
\nc\cS{\mathcal{S}}
\nc\cT{\mathcal{T}}
\nc\cU{\mathcal{U}}
\nc\cV{\mathcal{V}}
\nc\cW{\mathcal{W}}
\nc\cY{\mathcal{Y}}
\nc\cX{\mathcal{X}}
\nc\cZ{\mathcal{Z}}

\nc{\dmo}{\DeclareMathOperator}
\dmo{\Tw}{Twist}
\dmo{\CP}{Pres}
\rnc{\Re}{\operatorname{Re}}
\rnc{\Im}{\operatorname{Im}}
\rnc{\span}{\operatorname{span}}
\dmo{\rank}{rank}
\dmo{\End}{End}
\dmo{\Hom}{Hom}
\dmo{\Jac}{Jac}
\dmo{\Id}{Id}
\dmo{\lcm}{lcm}
\dmo{\Area}{Area}

\nc{\Tm}{Teichm\"uller\xspace}
\nc{\odd}{\cH^{\text{odd}}(4)}
\nc{\hyp}{\cH^{\text{hyp}}(4)}
\nc{\prym}{\tilde{\mathcal{Q}}(3,-1^3)}
\nc{\G}{GL^+(2,\bR)}
\nc{\GL}{GL^+}

\title[Translation surfaces and their orbit closures]{Translation surfaces and their orbit closures: \\An introduction for a broad audience}
\author[Wright]{Alex~Wright}
\begin{document}
\maketitle



Translation surfaces can be defined in an elementary way via polygons, and arise naturally in in the study of various basic dynamical systems. They can also be defined as  differentials  on Riemann surfaces, and have moduli spaces called strata that are related to the moduli space of Riemann surfaces.  There is a $GL(2,\bR)$ action on each stratum, and to solve most problems about a translation surface one must first know the closure of its orbit under this action. Furthermore, these orbit closures are of fundamental interest in their own right, and are now known to be algebraic varieties that parameterize translation surfaces with extraordinary algebro-geometric and flat properties. The study of  orbit closures has greatly accelerated in recent years, with an influx of new tools and ideas coming diverse areas of mathematics.  

Many areas of mathematics, from algebraic geometry and number theory, to dynamics and topology, can be brought to bear on this topic, and furthermore known examples of orbit closures are interesting from all these points of view. 

This survey is an invitation for mathematicians from different backgrounds to become familiar with the subject. Little background knowledge, beyond the definition of a Riemann surface and its cotangent bundle, is assumed, and top priority is given to presenting a view of the subject that is at once  accessible and connected to many areas of mathematics. 

\bold{Acknowledgements:} I am especially grateful to Ronen Mukamel for helpful conversations that shaped the presentation of real multiplication in genus 2, and to Curtis McMullen for helpful comments and corrections. I am also very grateful to Preston Wake, Clark Butler, Elise Goujard, Emre Sertoz, Paul Apisa, Ian Frankel, Benjamin Dozier, Zhangchi Chen, and Zhan Jiang for finding typos and making helpful comments on these notes. 

This survey grew out of notes for a five lecture course at the Graduate Workshop on Moduli of Curves, held July 7--18, 2014 at the Simons Center in Stony Brook, NY, organized by Samuel Grushevsky, Robert Lazarsfeld, and Eduard Looijenga. The author is grateful to the organizers and participants. 

The author is partially supported by a Clay Research Fellowship. 

\bold{Other resources.} There are a number of very good surveys on translation surfaces and related topics, for example \cite{Esur,Fsur, HSsur, Masur, MT, MPC, MAff, SWprobs, Vi, Yo1, Yo2, Yo3, Z}. See also the seminal paper \cite{Mc} of McMullen.

\section{Translation surfaces}\label{S:intro}

\subsection{Equivalent definitions} This  subsection has been written in a fairly technical way, so that it may serve as a reference for anyone looking for details on foundational issues. Most readers will want to skip some of the proofs on first reading. Anyone who already knows what a translation surface is can skip this subsection entirely. 

In these notes, all Riemann surfaces will  be assumed to be compact and connected. (A Riemann surface is a manifold of real dimension two with an atlas of charts to $\bC$ whose transition maps are biholomorphic.) Thus the term ``Riemann surface" will be synonymous to ``irreducible smooth projective algebraic curve over $\bC$."

\begin{defn}
An \emph{Abelian differential} $\omega$ on a Riemann surface $X$ is a global section of the cotangent bundle of $X$. 

 A \emph{translation surface} $(X,\omega)$ is a nonzero Abelian differential $\omega$ on a Riemann surface $X$. 
\end{defn}

Thus ``nonzero Abelian differential" and ``translation surface" are synonymous terms, but sometimes the notation is slightly different. Sometimes we might omit the word ``translation," and say ``let $(X,\omega)$ be a surface" when we should say ``let $(X,\omega)$ be a translation surface."


The complex vector space of Abelian differentials on $X$ will be denoted $H^{1,0}(X)$. We assume the following facts are familiar to the reader. 

\begin{thm}
Let $g$ denote the genus of $X$. Then $\dim_\bC H^{1,0}(X) = g$. If $g>0$, each Abelian differential $\omega$ on $X$ has $2g-2$ zeros, counted with multiplicity.

Each nonzero Abelian differential is a 1-form, which is closed but not exact, and hence $H^{1,0}$ is naturally a subspace of the first cohomology group $H^1(X,\bC)$ of $X$.  
\end{thm}

The following result is key to how most people think about translation surfaces. 

\begin{prop}\label{P:standard}
Let $(X,\omega)$ be a translation surface. At any point where $\omega$ is not zero, there is a local coordinate $z$ for $X$ in which $\omega=dz$. At any point where $\omega$ has a zero of order $k$, there is a local coordinate $z$ in which $\omega=z^kdz$. 
\end{prop}

\begin{proof}
We will work in a local coordinate $w$, and suppose that $\omega$ vanishes to order $k$ at $w=0$. Thus we can write $\omega=w^k g(w)$, where $g$ is some holomorphic function with $g(0)\neq 0$. Note that 
$$\int_{0}^{w} g(t)t^k dt$$
vanishes to order $k+1$ at $0$, and thus admits a $(k+1)$-st root. Define $z$ by 
$$z^{k+1}=(k+1)\int_{0}^{w} g(t)t^k dt.$$
By taking $d$ of each side, we find $z^k dz = \omega$ as desired. 
\end{proof}

Let $\Sigma\subset X$ denote the finite set of zeros of $\omega$. At each point $p_0$ of $X\setminus \Sigma$, we may pick a local coordinate $z$ as above. This choice is unique if we require $z(p_0)=0$, and otherwise it is unique up to translations. This is because if $f(z)$ is any holomorphic function with $df=dz$, then $f(z)=z+C$ for some constant $C$. This leads to the following.

\begin{prop}
$X\setminus \Sigma$ admits an atlas of charts to $\bC$ whose transition maps are translations. 
\end{prop}

\begin{proof}
The atlas consists of all local coordinates $z$ with the property that $\omega=dz$. 
\end{proof}

In particular, this gives $X\setminus \Sigma$ the structure of a flat manifold, since translations preserve the standard flat (Euclidean) metric on $\bC$. However, there is even more structure: for example, at every point there is a choice of ``north" (the positive imaginary direction). 

We will see that the flat metric on $X\setminus \Sigma$ does not extend to a flat metric on $X$. This should be reassuring, since for us typically $X$ will have genus at least 2, and the Gauss-Bonnet Theorem implies that such surfaces do not admit (nonsingular) flat metrics.

The points of $\Sigma$ are thus  considered to be singularities of the flat metric. From now on the term ``singularity" of $(X,\omega)$ will be synonymous with ``zero of $\omega$." The singularity is said to be order $k$ if $\omega$ vanishes to order $k$. 

We are now left with the task of determining what the flat metric looks like in a neighbourhood of a singularly $p_0$ of order $k$. We may use a local coordinate $z$ where $z(p_0)=0$ and $\omega=(k+1)z^kdz$ (this is a scalar multiple of the local coordinate constructed above). The 1-form $(k+1)z^kdz$ is the pull back of the form $dz$ via the branched covering map $z\mapsto z^{k+1}$, since $d(z^{k+1})=(k+1)z^kdz$. Near every point near but not equal to $p_0$,  $w=z^{k+1}$ is a local coordinate in which $\omega=dw$. Thus the flat metric near these point is the pull back of the flat metric on $\bC$ under this map $z\mapsto z^{k+1}$.

This pull back metric may be thought of very concretely: take $(k+1)$ copies of the upper half plane with the usual flat metric, and $(k+1)$ copies of the lower half plane, and glue them along the half infinite rays $[0,\infty)$ and $(-\infty,0]$ in cyclic order as in figure \ref{F:zero1}. 
\begin{figure}[h!]
\centering
\includegraphics[scale=0.22]{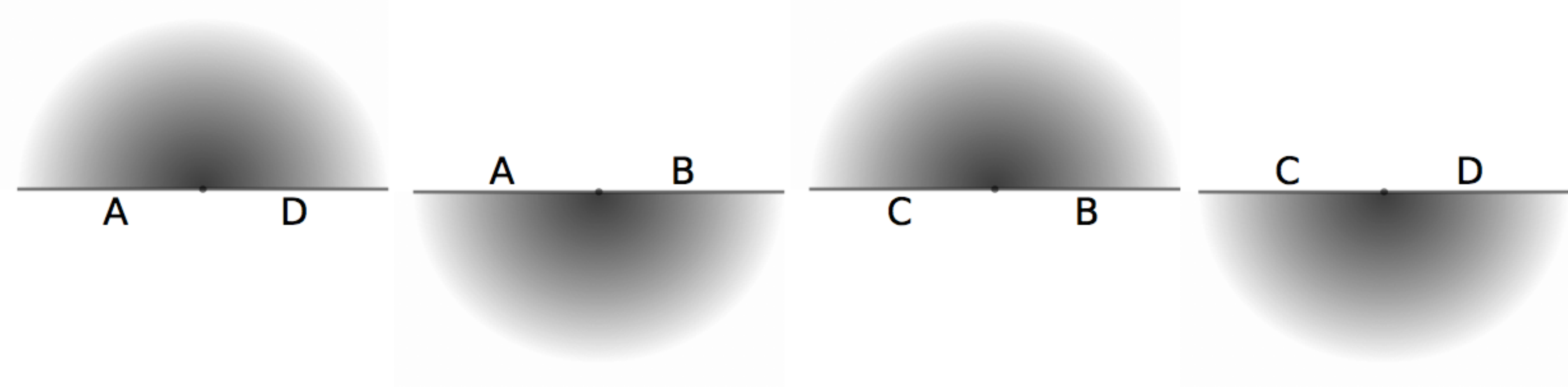}
\caption{Four half planes glued in cyclic order. A neighbourhood of any singularity of order 1 is isometric to a neighbourhood of 0 in the picture. }
\label{F:zero1}
\end{figure}

\begin{defn}[Second definition of translation surface]
A \emph{translation surface} is a closed topological surface $X$, together with a finite set of points $\Sigma$ and an atlas of charts to $\bC$ on $X\setminus \Sigma$ whose transition maps are translations, such that at each point $p_0$ of $\Sigma$ there is some $k>0$ and a homeomorphism of a neighborhood of  $p_0$  to a neighbourhood of the origin in the $2k+2$ half plane construction that is an isometry away from $p_0$. 
\end{defn}

The singularity at $p_0$ is said to have cone angle $2\pi(k+1)$, since it can be obtained by gluing $2k+2$ half planes, each with an angle of $\pi$ at the origin. The term ``cone point" is another synonym for ``singularity."

\begin{prop}
The first and second definition of translation surface are equivalent. 
\end{prop}  

We have already shown that the structure in the first definition leads to the structure in the second definition, so it now suffices to show the converse. 

\begin{proof}
Given such a flat structure on a surface $X$ as in the second definition, we get an atlas of charts to $\bC$ away from the singularities, whose transition functions are translations. Since translations are biholomorphisms, this provides $X\setminus \Sigma$ with a complex structure, where $\Sigma$ is the set of singularities. Furthermore, we get an Abelian differential on $X\setminus \Sigma$, by setting $\omega=dz$ for any such local coordinate $z$. At each singularity $p_0$ of order $k$ of the flat metric, we can find a unique coordinate $z$ such that $z(p_0)=0$ and the  covering map $z^{k+1}/(k+1)$ is a local isometry (except at the point $p_0$) to a neighbourhood of $0$ in $\bC\setminus \{0\}$. In this coordinate $z$, the calculations above show that $\omega= z^k dz$ on a neighbourhood of $p_0$ minus $p_0$ itself ($\omega$ has not yet been defined at $p_0$). 

As soon as we check that the remaining transition maps are biholomorphic, we will conclude that this  atlas of charts (given by the coordinates $z$ as above) on $X$ gives $X$ a complex structure. Setting $\omega=z^k dz$ at each singularity, in the local coordinate above, completes the definition of the Abelian differential $\omega$. 

The transition maps can be explicitly computed. Suppose $z$ and $w$ are local coordinates, such that $\omega=z^k dz$ in a neighborhood of a singularity of a singularity $p_0$, and $\omega=dw$ in a smaller open subset not containing the singularity. Then there is some constant $C$ such that 
$$w(z)=C+\int_0^z \eta^k d\eta.$$ 
This is evidently a local biholomorphism away from $z=0$. 
\end{proof}

The third definition is the most concrete, and is how translation surfaces are usually given.

\begin{defn}[Third definition of translation surface]
A \emph{translation surface} is an equivalence class of polygons in the plane with edge identification: Each translation surface is a finite union of polygons  in $\bC$, together with a choice of pairing of parallel sides of equal length that are on ``opposite sides." (So for example two horizontal edges of the same length can be identified only if one is on the top of a polygon, and one is on the bottom. Each edge must be paired with exactly one other edge. These conditions are exactly what is required so that the result of identifying pairs of edges via translations is a closed surface.) Two such collections of polygons are considered to define the same translation surface if one can be cut into pieces along straight lines and these pieces can be translated and re-glued to form the other collection of polygons. When a polygon is cut in two, the two new boundary components must be paired, and two polygons can be glued together along a pair of edges only if these edges are paired. 
\end{defn}

\begin{figure}[h!]
\centering
\includegraphics[scale=0.8]{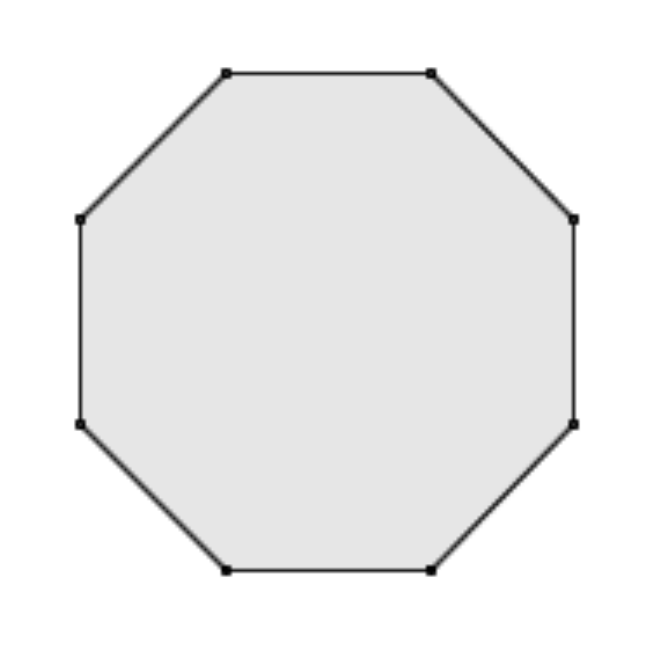}
\caption{When opposite edges of a regular octagon are identified, the result is a translation surface with one cone point of angle $6\pi$. (Generally the identifications are not drawn when opposite edges are identified--this situation is so common that it is the default.) A Euler characteristic computation shows that this has genus 2. ($2-2g=V-E+F$, where $g$ is the genus, $V$ is the number of vertices, $E$ is the number of edges, and $F$ is the number of faces. In this example, after identification of the edges there is 1 vertex, 4 edges, and 1 face, so $2-2g=1-4+1=-2$.)}
\end{figure}

\begin{figure}[h!]
\centering
\includegraphics[scale=0.4]{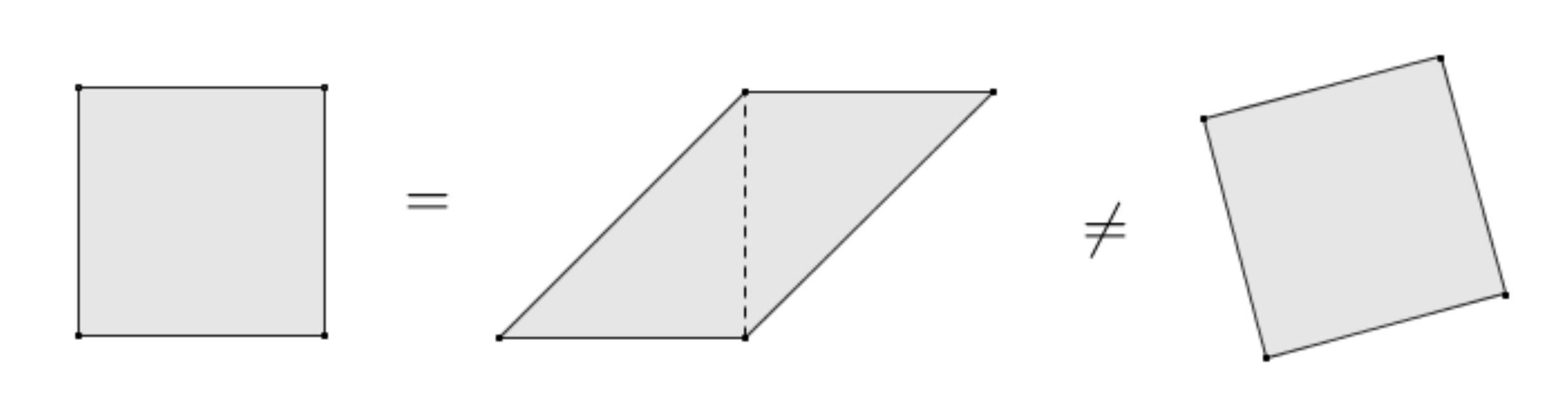}
\caption{In each of these three polygons, opposite edges are identified to give a genus one translation surface.  The first two are the same surface, since the second polygon can be cut (along the dotted line) and re-glued to give the first. However, the third translation surface is not equal to the first two, even though it is flat isometric. There is no flat isometry between them that sends ``north" (the positive imaginary direction) to ``north."}
\label{F:Tori}
\end{figure}

\begin{prop}\label{P:23}
The third definition of translation surface is equivalent to the second. 
\end{prop}

We will sketch the proof, but first some definitions are required. 

\begin{defn}
A \emph{saddle connection} on a translation surface is a straight line segment (i.e., a geodesic for the flat metric) going from a singularity to a singularity, without any singularities in the interior of the segment. (The two endpoints can be the same singularity or different.) The \emph{complex length} (also known as the \emph{holonomy}) of a saddle connection is the associated vector in $\bC$, which is only defined up to multiplication by $\pm1$. 

A \emph{triangulation} of a translation surface is a collection of saddle connections whose interiors are disjoint, and such that any component of the complement is a triangle. 
\end{defn}

\begin{rem}
We will not discuss general geodesics for the singular flat metric, except to remark in passing that saddle connections are examples, and a general geodesic is a sequence of saddle connections such that each forms an angle of at least $\pi$ with next, or a isometrically embedded circle (i.e., the core curve of a cylinder). In a certain definite sense,  most flat geodesics contain more than one saddle connection. 
\end{rem}

\begin{lem}
Every translation surface (using the second definition) can be triangulated. 
\end{lem}

\begin{proof}[Sketch of proof]
In fact, any maximal collection of saddle connections whose interiors are disjoint must be a triangulation. 
\end{proof}

\begin{proof}[{Sketch of proof of Proposition \ref{P:23}}]
In one direction, the lemma says that every surface as in the second definition can be triangulated. Cutting each saddle connection in the triangulation gives a collection of polygons (triangles) with edge identification. Two edges are identified if they were the same saddle connection before cutting. 

In the other direction, given a collection of polygons as in the third definition, the paired edges may be identified via translations. At each point on the interior of a polygon, the natural coordinate $z$ of $\bC$ can be used. (The polygon sits in the complex plane $\bC$.) At any point on the interior of an edge, the two polygons can be glued together, giving locally a coordinate. The structure at the singularities can be verified in an elementary way. 

Indeed, after identifying pairs of edges, some of the vertices will become singular points of the flat metric. The main point is that the total angle around these singularities is an integer multiple of $2\pi$. If the total angle were anything else, there would be no well defined choice of ``north." See figure \ref{F:Pacman}.
\end{proof}

\begin{figure}[h!]
\centering
\includegraphics[scale=0.2]{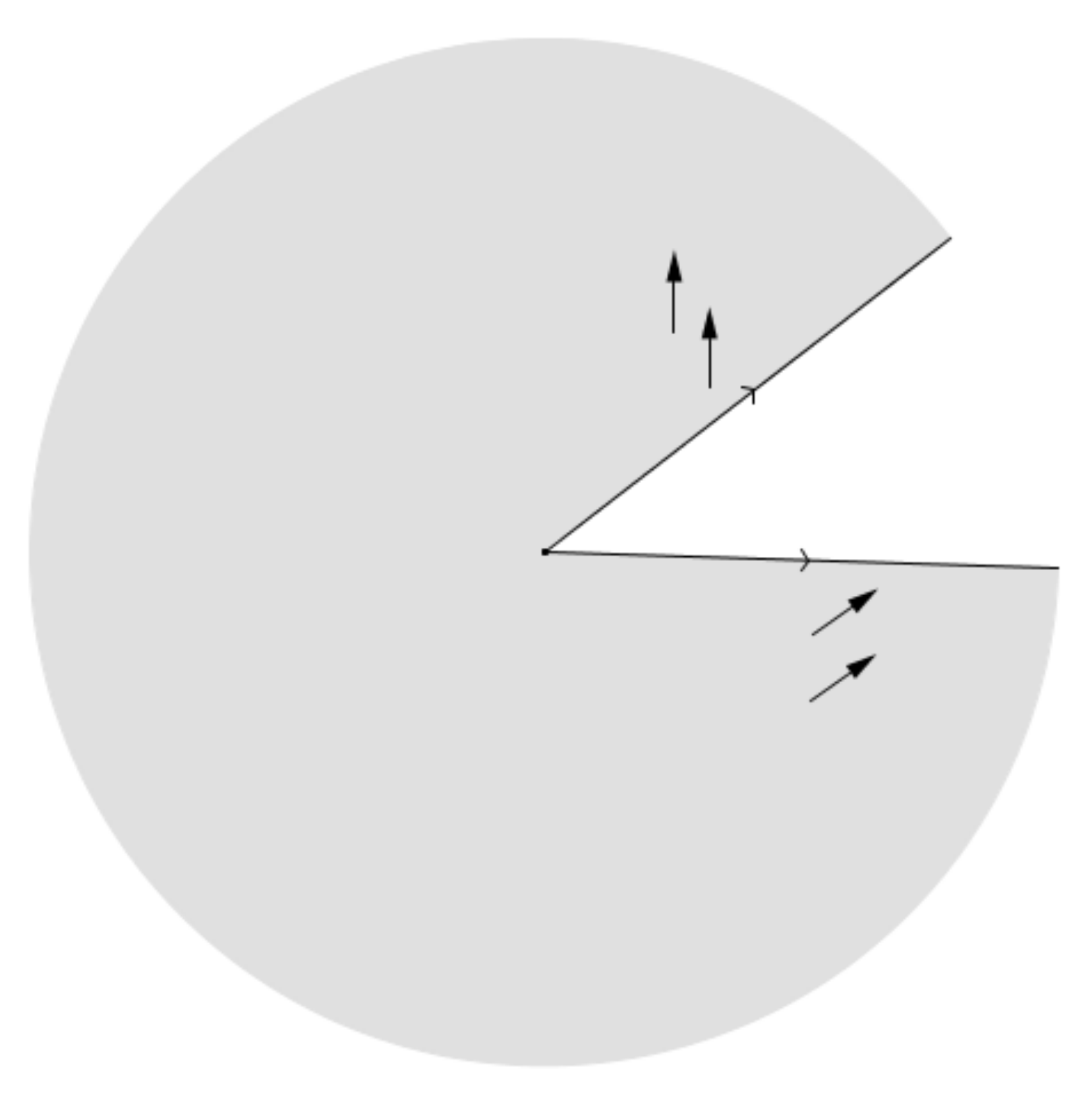}
\caption{Here is a model of a cone angle that is not an integer multiple of $2\pi$ (here it is less than $2\pi$), and hence cannot occur on a translation surface. It is obtained by identifying the two radial segments via rotation. On this picture, there is not a consistent choice of north: if a northward pointing vector is dragged across the radial segment, it no longer points north. }
\label{F:Pacman}
\end{figure}

\subsection{Examples} The flat torus $(\bC/\bZ[i], dz)$ is a translation surface.  This is pictured in figure \ref{F:Tori}.

\begin{defn}
A translation covering  $f:(X,\omega)\to(X', \omega')$ between translation surfaces is a branched covering of Riemann surfaces $f:X\to X'$ such that $f^*(\omega')=\omega$. 
\end{defn}

Translation coverings are, in particular, local isometries away from the ramification points. (By definition, the ramification points are the preimages of the branch points.) They must also preserve directions: for example, ``north" must map to ``north."  The ramification points must all be singularities. An unramified point is a singularity if and only if its image under the translation covering is. Branch points may or may not be singularities. 

The fact that translation coverings are local isometries away from ramification points is especially clear if one notes that the flat length of a tangent vector $v$ to the translation surface is $|\omega(v)|$.

\begin{defn}
A translation covering of $(\bC/\bZ[i], dz)$ branched over a single point is called a \emph{square-tiled surface}. 
\end{defn}

\begin{figure}[h!]
\centering
\includegraphics[scale=0.4]{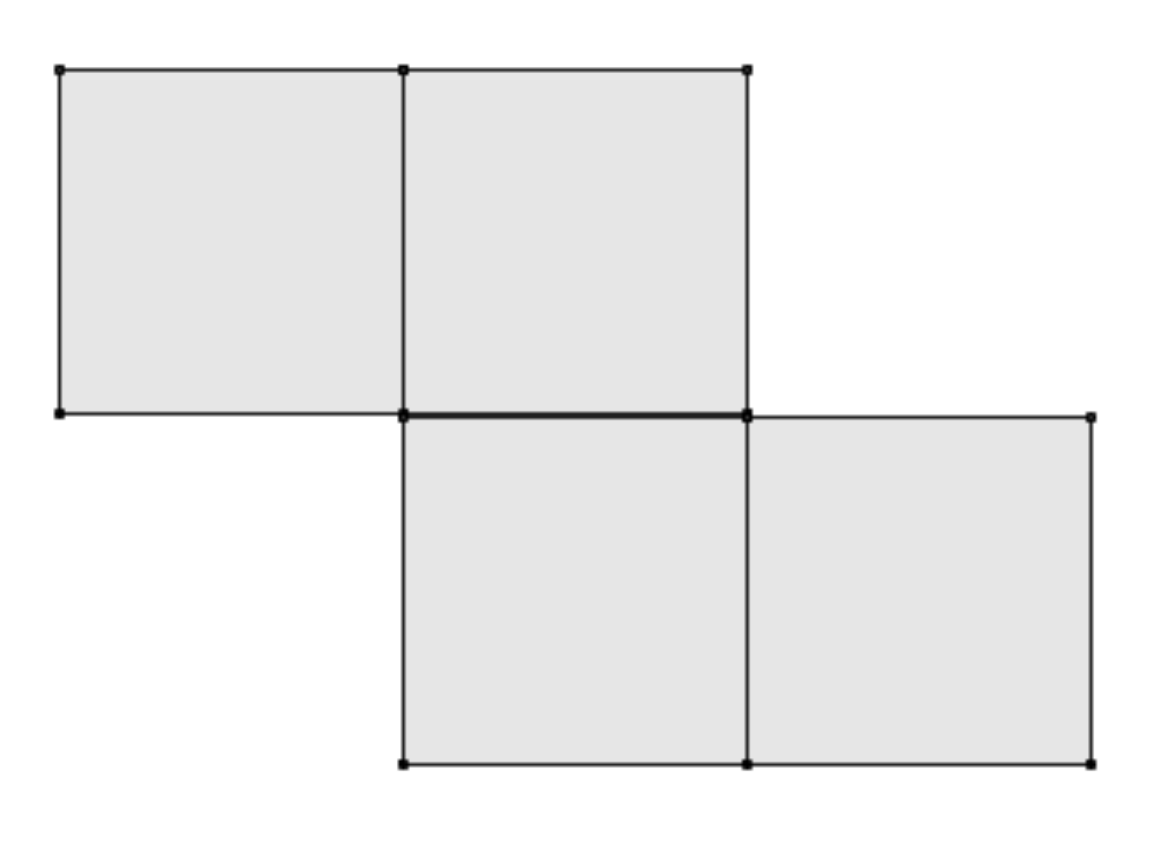}
\caption{An example of a square-tiled surface. Opposite edges are identified. This translation surface is a degree 4 cover of $(\bC/\bZ[i], dz)$ branched over 1 point. It is genus 2, and has two singularities, each of cone angle $4\pi$.}
\label{F:STgenus2}
\end{figure}

Indeed, $(\bC/\bZ[i], dz)$ is a square with opposite sides identified, and the branch point can be assumed to be the corners of the square. The square-tiled surface will be tiled by $d$ lifted copies of this square, where $d$ is the degree. 

\bold{The slit torus construction.} In this construction, one starts with two genus one translation surfaces, and picks a parallel embedded straight line segment on each of them, of the same length. The two segments are cut open, and the resulting tori with boundary are glued together. 
\begin{figure}[h!]
\centering
\includegraphics[scale=0.4]{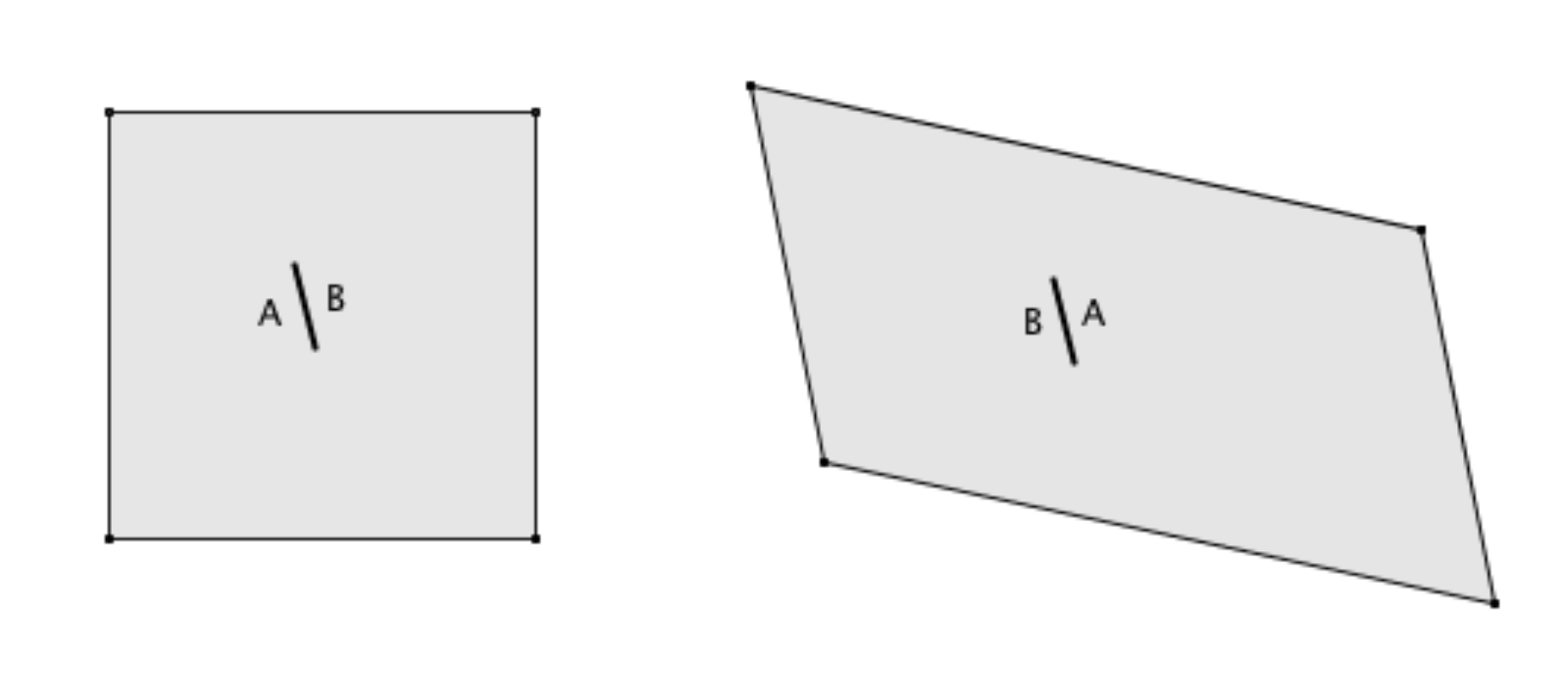}
\caption{An example of the slit tori construction. Opposite edges are identified, giving two tori. In each, a slit is made, so that each torus has boundary consisting of two line segments, labelled A and B here. These are glued together to give a translation surface of genus two with two singularities of cone angle $4\pi$, one at each end of the slit.}
\label{F:SlitTori}
\end{figure}

\bold{Unfolding rational billiards.} Perhaps the original motivation for translation surfaces came from the study of rational billiards. This will be explained in section 4. For now, begin with a polygon $P$ in $\bC$, all of whose angles are rational multiples of $\pi$. This restriction is equivalent to saying that the subgroup $H$ of $O(2)$ generated by the derivatives of reflections in the sides of the polygon is a finite group. (It is a dihedral group.)  

For each $h\in H$, we consider $hP$. We translate these if necessary so that the collection $\{hP: h\in H\}$ is a finite collection of disjoint polygons.  We identify the edges in pairs: if $h'P$ is the reflection of $hP$ in an edge of $hP$, this edge and the corresponding edge of $h'P$ are identified. 
\begin{figure}[h!]
\centering
\includegraphics[scale=0.25]{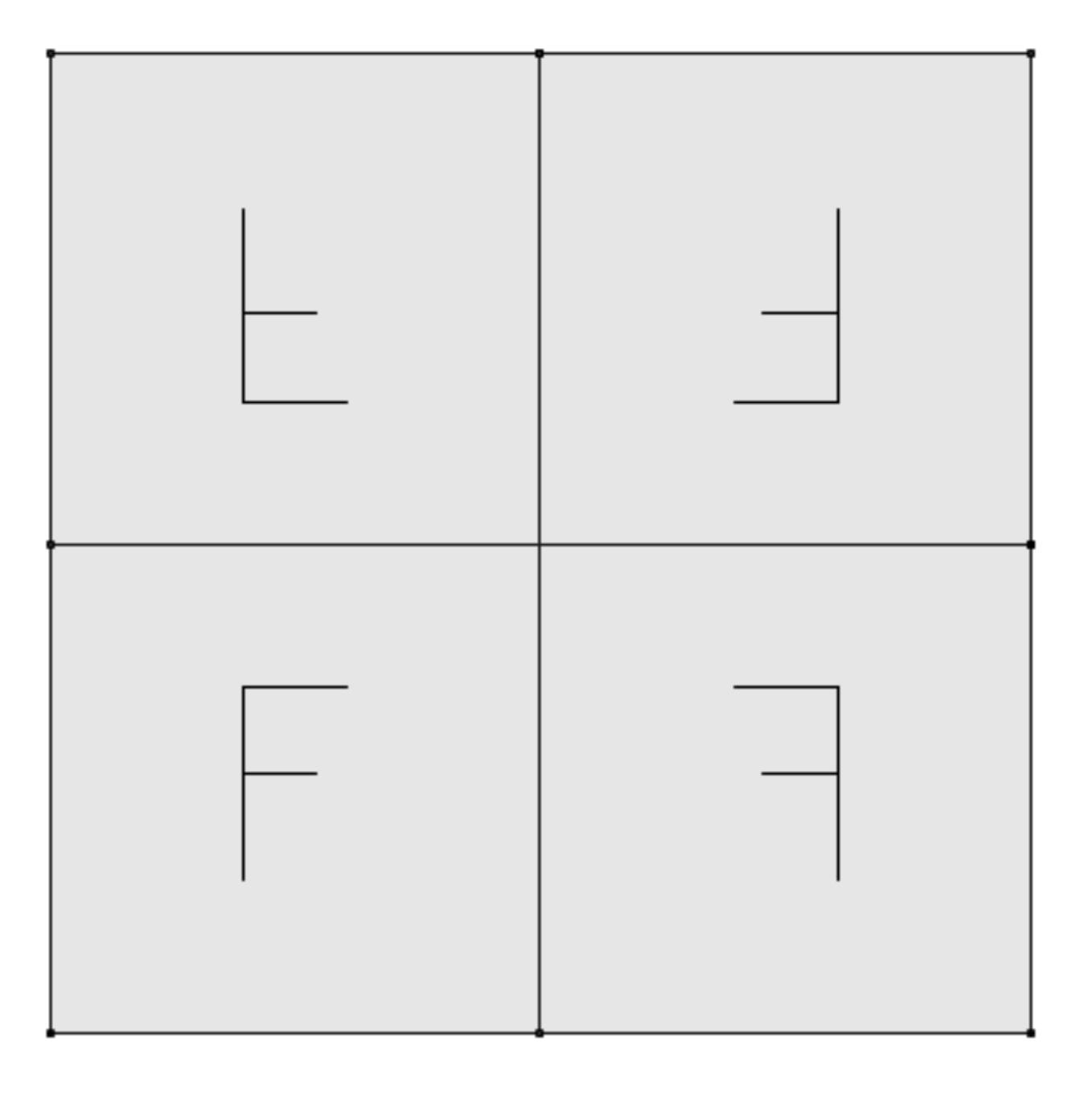}
\caption{The unit square unfolds to four squares, glued together to give the flat torus $(\bC/2\bZ[i], dz)$. Here each square has been decorated by the letter F, to illustrate which squares are reflections of other squares.}
\end{figure}
\begin{figure}[h!]
\centering
\includegraphics[scale=0.2]{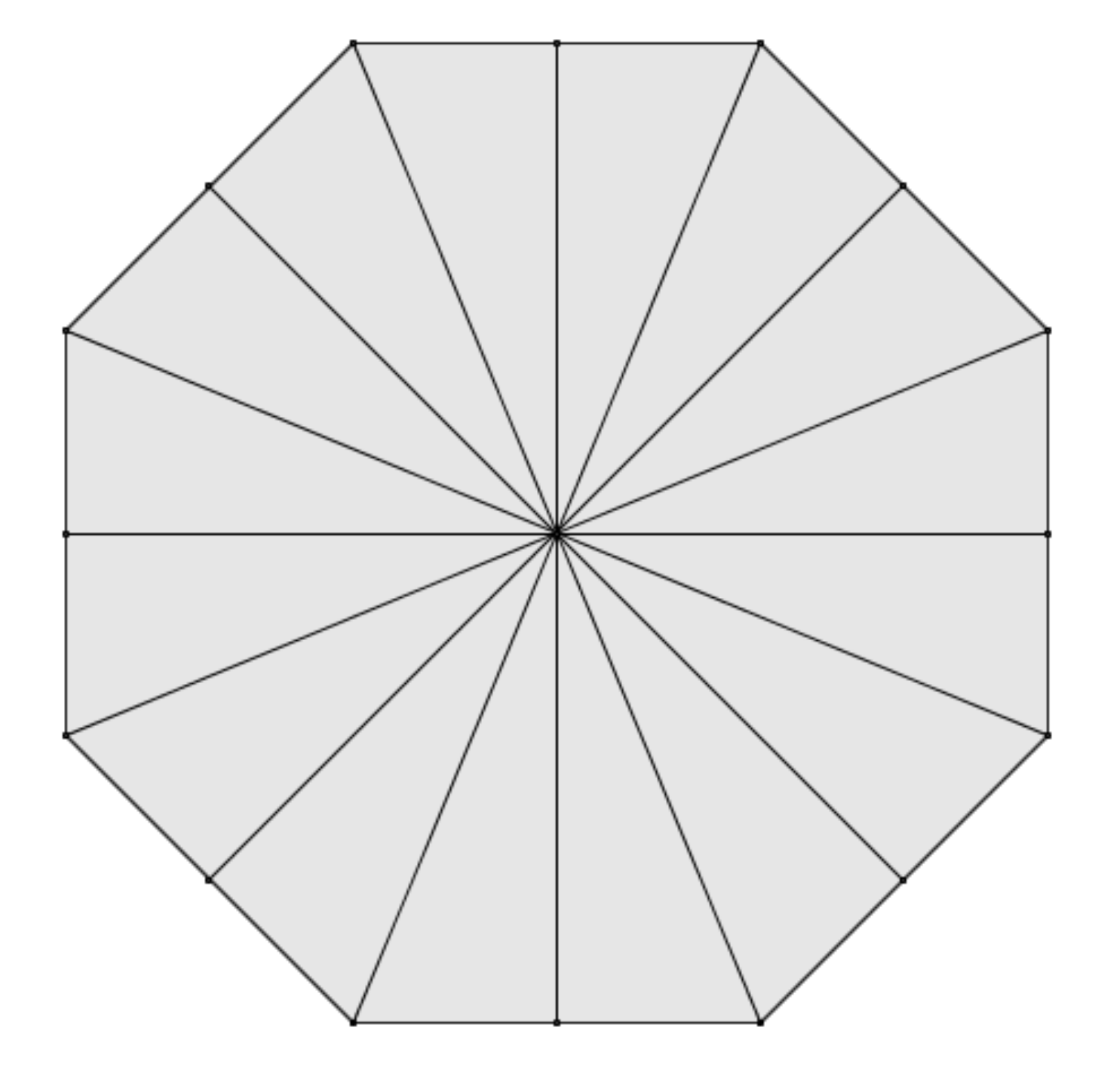}
\caption{Unfolding the right angled triangle with smallest angle $\pi/8$ gives the regular octagon with opposite sides identified.}
\end{figure}

\begin{prop}\label{P:pqr}
Suppose $\gcd(p,q,r)=1$ and $p+q+r=n$. The triangle with angles $\frac{p}{n}\pi, \frac{q}{n}\pi, \frac{r}{n}\pi$ unfolds to the Abelian differential $\frac{ydx}{x(x-1)}$ on the normalization of the algebraic curve 
$$y^n=x^p(x-1)^q.$$
\end{prop}

This curve is a cyclic cover of $\bC P^1$ via the map $(y,x)\mapsto x$. 

The proof of this proposition is short and uses only classical complex analysis (Schwarz-Christoffel mappings), see for example \cite{DT}.

\subsection{Moduli spaces} Consider the problem of deforming the regular octagon. We may specify four of the edges as vectors in $\bC$, and thus guess (correctly!) that the moduli space is locally $\bC^4$. 

\begin{figure}[h!]
\centering
\includegraphics[scale=0.4]{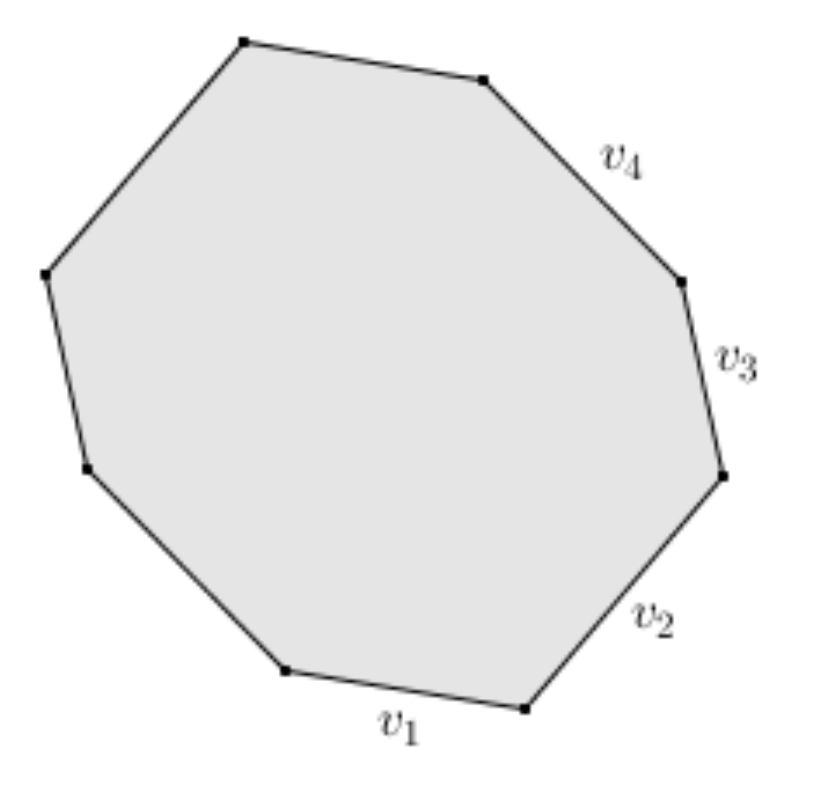}
\caption{An octagon with opposite edges parallel may be specified by four complex numbers $v_1, v_2, v_3, v_4\in \bC$. (Not all choices give a valid octagon without self crossings, but there is an open set of valid choices.)}
\end{figure}

An Abelian differential in genus two can have either a double zero, or two  zeros of order one. The octagon has a double zero, and deformations as above will always  have a double zero. Figures \ref{F:STgenus2} and \ref{F:SlitTori} illustrate genus two translation surfaces with two zeros of order one.

The collection of all Abelian differentials of genus $g$ is of course a vector bundle over the moduli space of Riemann surfaces. However, this space is stratified according to the number and multiplicity of the zeros of the Abelian differentials.

Let $g>1$ and let $\kappa$ denote a partition of $2g-2$, i.e. a nonincreasing list of positive integers whose sum is $2g-2$. So if $g=2$, the partitions are $(2)$ and $(1,1)$, and if $g=3$ the partitions are $(4), (3,1), (2,2), (2,1,1), (1,1,1,1)$. 

Define the stratum $\cH(\kappa)$ as the collection of genus $g$ translations surfaces $(X,\omega)$, where the multiplicity of the zeros of $\omega$ are given by $\kappa$. So $\cH(2)$ denotes the collection of genus two translation surfaces with a double zero, and $\cH(1,1)$ denotes the collection of all genus two translation surfaces with two simple zeros. 

\begin{prop}
Each stratum $\cH(\kappa)$ is a complex orbifold of dimension $n=2g+s-1$, where $s=|\kappa|$ denotes the number of distinct zeros of Abelian differentials in the stratum. Away from orbifold points (or on an appropriate  cover without orbifold points) each stratum has an atlas of charts to $\bC^n$ with transition functions in $GL(n,\bZ)$. 
\end{prop}

Thus each stratum looks locally like $\bC^n$, and has a natural affine structure. 

\begin{proof}[Sketch of formal proof]
Let $S$ be fixed topological surface of genus $g$, with a set $\Sigma$ of  $s$  distinct marked points. Let us begin with the space $\tilde{\cH}(\kappa)$ of translations surfaces $(X,\omega)$ equipped with an equivalence class of homeomorphisms $f:S\to X$ that send the marked points to the zeros of $\omega$. The equivalence relation is isotopy rel marked points. 

We will see that the map from $\tilde{\cH}(\kappa)$ to $\cH(\kappa)$ that forgets $f$ is an infinite degree branched covering. 

Fix a basis for the relative homology group $H_1(S, \Sigma, \bZ)$. If $\Sigma=\{p_1, \ldots, p_n\}$, this is typically done by picking a (symplectic) basis $$\gamma_1, \ldots, \gamma_{2g}$$ for absolute homology $H_1(S, \bZ)$, and then picking a curve $\gamma_{2g+i}$ from $p_i$ to $p_n$ for each $i=1, \ldots, n-1$. The map
$$\tilde{\cH}(\kappa)\to \bC^n, \quad\quad (X,\omega, [f]) \mapsto \left( \int_{f_*\gamma_i}\omega\right)_{i=1}^{2g+s-1}$$
is locally one-to-one and is onto an open subset of $\bC^n$. The easiest way to see this is via flat geometry: These integrals determine the integrals of every relative homology class, and in particular the complex lengths of the edges in any polygon decomposition for $(X,\omega)$. The edges in this polygon decomposition of course determine $(X,\omega)$. 

The mapping class group of $S$ (with $s$ distinct unlabeled marked points) acts on $\tilde{\cH}(\kappa)$ by precomposition of the marking. The induced action on these $\bC^n$ coordinates is via $GL(n,\bZ)$ (change of basis for relative homology). The quotient is $\cH(\kappa)$.
\end{proof}

In this course we will ignore orbifold issues, and just pretend strata are complex manifolds rather than complex orbifolds. Given a translation surface $(X,\omega)$, and a basis $\gamma_i$ of the relative homology group $H_1(X,\Sigma, \bZ)$, we will refer to 
$$\left( \int_{\gamma_i}\omega\right)_{i=1}^{2g+s-1}\in \bC^n$$
as \emph{period coordinates} near $(X,\omega)$. Implicit in this is that the basis can be canonically transported to nearby surfaces in $\cH(\kappa)$, thus giving  a map from a neighbourhood of $(X,\omega)$ to a neighbourhood in $\bC^n$. However, when $X$ has automorphisms preserving $\omega$, this is not the case canonically. This is precisely the issue we are ignoring when we pretend that $\cH(\kappa)$ is a manifold instead of an orbifold. 

It is precisely the period coordinates that provide strata an atlas of charts to $\bC^n$ with transition functions in $GL(n,\bZ)$. The transition functions are change of basis matrices for relative homology $H_1(X,\Sigma, \bZ)$. 

\bold{Compactness criterion.} Strata are never compact. Even the subset of unit area translation surface is never compact. (Here and throughout these notes we refer to the analytic topology, which is the weakest topology for which period coordinates are continuous.)

Masur's compactness criterion gives that a closed subset of the set of unit area surfaces in a stratum is compact if and only if there is some positive lower bound for the length of all saddle connections on all translation surfaces in the subset. 

Let us also remark that the map $(X,\omega)\mapsto X$ is not proper, even when restricting to unit area surfaces. For example, it is possible to have a sequence of translation surfaces $(X_n,\omega_n)$ of area 1 in $\cH(1,1)$ converge to $(X,\omega)\in \cH(2)$ (so two zeros coalesce) in the bundle of Abelian differentials over the moduli space of Riemann surfaces. However, such a sequence $(X_n,\omega_n)$ will diverge in $\cH(1,1)$: there will be shorter and shorter saddle connections joining the two zeros. 

\bold{Every Abelian differential in genus two.} We now give flat geometry pictures of all Abelian differentials in genus 2. The discussion includes only sketches of proofs. 

\begin{prop}
Every translation surface in $\cH(2)$ can be obtained by gluing a cylinder into a slit torus, as in figure \ref{F:AllH2}. Every translation surface in $\cH(1,1)$ is obtained from the slit torus construction, as in figure \ref{F:AllH11}
\end{prop}

\begin{lem}
For any translation surface in genus 2, not every saddle connection is fixed by the hyperelliptic involution $\rho$. 
\end{lem}

\begin{proof}
Triangulate the surface. If every saddle connection in this triangulation was fixed by the hyperelliptic involution, then each triangle would be mapped to itself. However, no triangle has rotation by $\pi$ symmetry, so this is impossible.  
\end{proof}

\begin{proof}[Sketch of proof of proposition.] Now fix a translation surface in $\cH(2)$, and a saddle connection $c$ not fixed by $\rho$. Since $\rho_*$ acts on homology by $-1$, $c$ and $\rho(c)$ are homologous curves. Cutting  $c$ and $\rho(c)$ decomposes the surface into two subsurfaces with boundary, one of genus one and one of genus zero. The genus zero component is a cylinder  and the genus one part must be a slit torus. 

\begin{figure}[h!]
\centering
\includegraphics[scale=0.5]{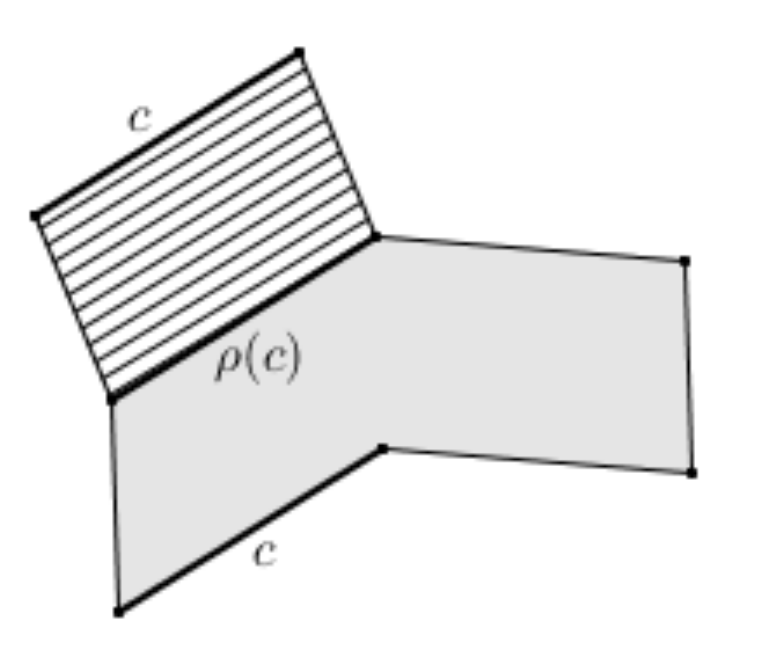}
\caption{Every translation surface in $\cH(2)$ admits a deposition into a cylinder and a slit torus, and hence can be drawn as in this picture. Opposite sides are identified.}
\label{F:AllH2}
\end{figure}

Similarly for a translation surface in $\cH(1,1)$, consider a triangulation. This must contain at least one triangle with one vertex at one of the singularities, and the other two vertices at the other singularity (i.e., for at least one of the triangles, not all corners are at the same singularity). At most one of the edges of this triangle is fixed by the hyperelliptic involution, so the triangle must have at least one edge $c$ that is not fixed by the hyperelliptic involution and goes between the two zeros. 

Cutting along $c$ and $\rho(c)$ decomposes the surface into two tori. In other words, every translation surface in $\cH(1,1)$ may be obtained from the slit torus construction. 

\begin{figure}[h!]
\centering
\includegraphics[scale=0.4]{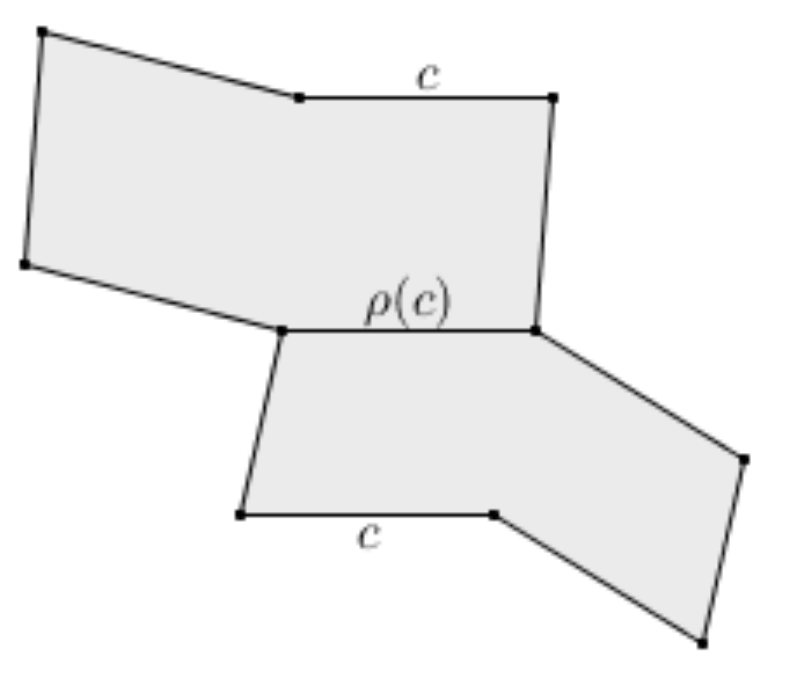}
\caption{Every translation surface in $\cH(1,1)$ comes from the slit torus construction, and thus can be drawn as in this picture.}
\label{F:AllH11}
\end{figure}

\end{proof}

\bold{Philosophical conclusions.} Polygon decompositions and flat geometry provide a fundamentally different way to think about Abelian differentials on complex algebraic curves. In this perspective an Abelian differential is extremely easy to write down (as a collection of polygons), and it is easy to visualize deformations of this Abelian differential (the edges of the polygons change). However, some things become more difficult in this perspective. For example, given a plane algebraic curve it is often an easy exercise to write down a basis of Abelian differentials. However, given a translation surface given in terms of polygons, this is typically impossible to do. And it is typically a transcendentally difficult problem to write down the equations for the algebraic curve given by some polygon decomposition. 

The study of translation surfaces and algebraic curves is  enriched by the transcendental connections between the two perspectives. 

\section{Affine invariant submanifolds}

\subsection{Definitions and first examples}  Fix a stratum, and a translation surface $(X,\omega)$ in the stratum. As always, let $\Sigma\subset X$ be the zeros of $\omega$. 

There is a linear injection $H^{1,0}(X) \to H_1(X,\Sigma, \bC)^*=H^1(X,\Sigma, \bC)$. Given $\omega\in H^{1,0}(X)$, the linear functional it determines in $H_1(X,\Sigma, \bC)^*$ is simply integration of $\omega$ over relative homology classes. Given a basis $\gamma_1, \ldots, \gamma_n$ of $H_1(X,\Sigma, \bZ)$, we get an isomorphism 
$$H^1(X,\Sigma, \bC)\to \bC^n, \quad\quad \phi \mapsto \left(\phi(\gamma_i)\right)_{i=1}^n.$$
The period coordinates of the previous section are just the composition of the map sending $(X,\omega)$ to $\omega$ considered as a relative cohomology class, followed by this isomorphism to $\bC^n$. About half the time it is helpful to forget this coordinatization, and just consider period coordinates as a map to $H^1(X,\Sigma, \bC)$ sending $(X,\omega)$ to the relative cohomology class of $\omega$. 

In the next definition we fix a stratum $\cH=\cH(\kappa)$, where $\kappa$ is a partition of $2g-2$ and $g$ is the genus. 

\begin{defn}
An \emph{affine invariant submanifold} is the image of a proper immersion $f$ of an open connected manifold $\cM$ to a stratum $\cH$, such that each point $p$ of $\cM$ has a neighbourhood $U$ such that in a neighbourhood of $f(p)$, $f(U)$ is determined by  linear equations in period coordinates with coefficients in $\bR$ and constant term $0$. 
\end{defn}

The difference between immersion and embedding is a minor technical detail, so for notational simplicity we will typically consider an affine invariant submanifold $\cM$ to be a subset of a stratum. The definition says that locally in period coordinates $\cM$ is a linear subspace of $\bC^n$. The requirement that the linear equations have coefficients in $\bR$ is equivalent to requiring that the linear subspace of $\bC^n$ is the complexification of a real subspace of $\bR^n$ (or, in coordinate free terms, of $H^1(X,\Sigma, \bR)$). 

The word ``affine" does not refer to affine varieties; it refers to the linear structure on $\bC^n$. (Perhaps ``linear" would be a better term than ``affine", but this is the terminology in use.)

\begin{figure}[h]
\centering
\includegraphics[scale=0.35]{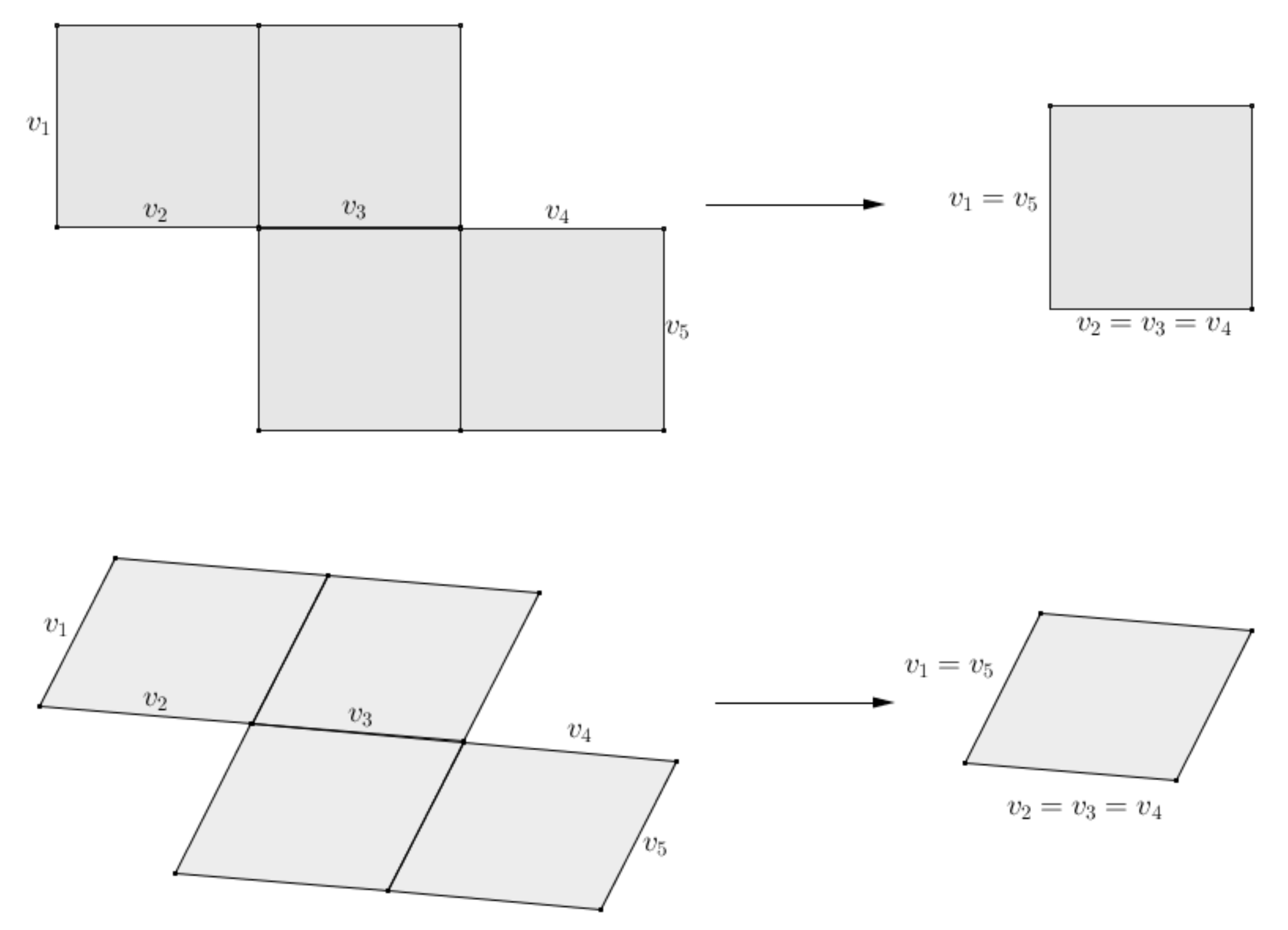}
\caption{Period coordinates $(v_1, v_2, v_3, v_4, v_5)\in \bC^5$ have been indicated on two surfaces in $\cH(1,1)$ (upper and lower left) that are degree four translation covers of tori (upper and lower right). The linear equations indicated locally cut out the locus of surfaces in $\cH(1,1)$ that are degree four translation covers of tori branched over 1 point.}
\label{F:EllBasis}
\end{figure}

In these notes, when we refer to the dimension of complex manifolds or vector spaces, we will mean their complex dimension.

Note that affine invariant submanifolds must have dimension at least $2$. This is because if $(X,\omega)$ is a point on an affine invariant submanifold, then the linear subspace must contain the real and imaginary parts the period coordinates, and these cannot be collinear. That they cannot be collinear follows from the fact that $\Re(\omega)$ and $\Im(\omega)$ cannot be collinear real cohomology classes, because 
$$\int_X \Re(\omega) \Im(\omega)$$
gives the area of the translation surface. (Locally $\omega=dz=dx+idy$, so $\Re(\omega) \Im(\omega)=dxdy$.) 

\begin{ex}\label{E:cover}
Let $\cM$ be a connected component of the space of degree $d$ translation covers of a genus one translation surface, which is allowed to vary, branched over $k$ distinct points. Then $\cM$ is an affine invariant submanifold of whichever stratum it lies in. 

Indeed, $\cM$ is a $(k+1)$-dimensional immersed manifold: the moduli space of genus one translation surfaces is 2-dimensional, and after one branched point is fixed at 0, the other $k-1$ are allowed to vary on the torus. So it suffices to see that locally in period coordinates, $\cM$ is contained in an $(k+1)$-dimensional linear subspace. 

Let $(X,\omega)\in \cM$, and say $f:(X,\omega)\to (X', \omega')$ is the translation covering, branched over $k$ points. Let $\Sigma\subset X$ be the set of zeros of $\omega$, and $\Sigma'\subset X'$ be the set of branch points of $f$, so $f(\Sigma)=\Sigma'$.

Let $\gamma\in \ker(f_*)\subset H_1(X,\Sigma, \bZ)$, and note that 
\begin{eqnarray*}
\int_\gamma \omega  
&=&
\int_\gamma f^*(\omega') 
\\&=&
\int_{f_*\gamma} \omega' =0.
\end{eqnarray*}
Thus since $\ker(f_*)$ has dimension $\dim_\bC H_1(X,\Sigma, \bC) - \dim_\bC H_1(X',\Sigma', \bC)$, we see that locally $\cM$ lies in an $\dim_\bC H_1(X',\Sigma', \bC)=(k+1)$-dimensional linear subspace in period coordinates, as desired. 

To rephrase the discussion in coordinate free terms, the relative cohomology class of $\omega$ lies in $f^*(H^1(X', \Sigma', \bC))$. 
\end{ex}

\begin{ex}
Fix a stratum $\cH$ and numbers $k$ and $d$, and let $\cM$ be a connected component of the space of all degree $d$ translation covers of surfaces in $\cH$ branched over $k\geq 0$ distinct points. Then similarly $\cM$ is an affine invariant submanifold. 

Similarly, if $\cM$ is an affine invariant submanifold, and $\cM'$ is a connected component of the space of all degree $d$ translation covers of surfaces in $\cM$ branched over $k\geq 0$ distinct points, then $\cM'$ is also an affine invariant submanifold. 
\end{ex}

\begin{ex}
Fix a stratum $\cH$ and a number $k$, and let $\cQ$ be a connected component of the locus of $(X,\omega)\in \cH$ where $X$ admits an involution $\iota$ with $\iota^*(\omega)=-\omega$ and $k$ fixed points. Then $\cQ$ is an affine invariant submanifold of $\cH$ (although it may be empty). It is locally defined by the equations 
$$\int_{\iota_* \gamma} \omega + \int_{\gamma} \omega =0,$$
for $\gamma\in H_1(X,\Sigma, \bZ)$. 

A special case of this example is when $\iota$ is the hyperelliptic involution. In some strata, there is a whole connected component of hyperelliptic surfaces; in some, there are none; and in some, the hyperelliptic surfaces form a proper affine invariant submanifold. 
\end{ex}

\subsection{Real multiplication in genus 2} We will begin with an elementary construction of some affine invariant submanifolds in genus 2. After we have constructed them, we will show that they parameterize Riemann surfaces $X$ whose Jacobian admits real multiplication with $\omega$ as an eigenform.

These affine invariant submanifolds were independently discovered by McMullen and Calta \cite{Mc, Ca} from very different perspectives. Our presentation is variant of McMullen's, and we suggest that the interested reader also consult \cite{Mc}.

For the next proposition, recall that $H^{1,0}(X)$ can be considered as a subspace of $H^1(X,\bC)$, and that there is a natural symplectic pairing $\langle \cdot, \cdot \rangle$ on $H^1(X,\bZ)$. An endomorphism $M$ of $H^1(X,\bZ)$ is called self-adjoint (with respect to this symplectic form) if $\langle M v, w \rangle=\langle v, M w \rangle$ for all $v,w\in H^1(X,\bZ)$. Note that different eigenspaces for a self-adjoint endomorphism must be symplectically orthogonal, since if $Mv=\lambda v$ and $Mw=\mu w$, then $$\lambda \langle v, w\rangle = \langle M v, w \rangle=\langle v, M w \rangle=\mu \langle v, w\rangle.$$

\begin{prop}\label{P:RMwoRM}
Fix an integer $D>0$ not a square. Consider the locus of $(X,\omega)$ in $\cH(1,1)$ or $\cH(2)$ for which there is a self-adjoint endomorphism $M:H^1(X, \bZ)\to  H^1(X, \bZ)$ whose extension to $H^1(X, \bC)$ satisfies $$M\omega=\sqrt{D}\omega.$$ In $\cH(1,1)$ this locus is a finite union of 3-dimensional affine invariant submanifolds, and in $\cH(2)$ it is a finite union of 2-dimensional affine invariant submanifolds. 
\end{prop}

In the proof, it is important to remember that because $\int_X \Re(\omega)\Im(\omega)>0$, the restriction of the symplectic form to the subspace $\span(\Re(\omega), \Im(\omega))$ is symplectic. 

\begin{proof}
First we will check that this locus is closed. (Recall that our default topology is the the analytic topology, which we are referring to here.) Suppose that $(X_n,\omega_n)$ is a sequence of surfaces in this locus, converging to $(X,\omega)$. We will first show that necessarily $(X,\omega)$ is in the locus, and hence the locus is closed. 

Let $M_n$ denote the endomorphism for $(X_n,\omega_n)$. This endomorphism $M_n$ has two $\sqrt{D}$-eigenvectors, $\Re(\omega_n)$ and $\Im(\omega_n)$. Since $M_n$ is an integer matrix, its $-\sqrt{D}$-eigenspace is the Galois conjugate of its $\sqrt{D}$-eigenspace, and hence must also be 2-dimensional. Since $M_n$ is self-adjoint, the $\sqrt{D}$ and $-\sqrt{D}$-eigenspaces are symplectically orthogonal. In particular, the $-\sqrt{D}$-eigenspace is the symplectic perp of the $\sqrt{D}$-eigenspace.

Note that $\omega_n$ converges to $\omega$ as cohomology classes. (The topology on strata is the topology on period coordinates, and period coordinates exactly determine the relative cohomology class of $\omega$.) Hence $M_n$ converges to an endomorphism $M$ of $H^1(X, \bZ)$, which acts by $\sqrt{D}$ on $\span(\Re(\omega), \Im(\omega))$ and by $-\sqrt{D}$ on the symplectic perp. Thus $(X,\omega)$ is in the locus also, and we have shown the locus is closed.

We must now show linearity. For this, it is important to note that above, since $\End(H^1(X, \bZ))$ is a discrete set, we see that $M_n$ is in fact eventually constant. So any $(X',\omega')$ in the locus close to $(X,\omega)$ has the same endomorphism.

Suppose $(X,\omega)$ is in the locus, with endomorphism $M$. Consider the set of $(X',\omega')$ sufficiently close to $(X,\omega)$ for which $M\omega'=\sqrt{D}\omega$. Since this equation is linear, we see that locally the locus is  linear. 

In other words, $\omega'$ must lie in the 2-dimensional $\sqrt{D}$-eigenspace in $H^1(X,\bC)$. In the 4-dimensional $\cH(2)$ the set of such $\omega'$ is 2-dimensional, however in the 5-dimensional $\cH(1,1)$, this set is 3-dimensional. (Period coordinates for $\cH(2)$ may be considered as a map to $H^1(X,\bC)$, and for $\cH(1,1)$ they may be considered as a map to $H^1(X, \Sigma, \bC)$. There is a natural map from $H^1(X, \Sigma, \bC)$ to $H^1(X,\bC)$, and we are requiring that the image of $\omega$ lies in a codimension 2 subspace there.)
\end{proof}

We can be more explicit about the linear equations that define the locus in period coordinates. Indeed, suppose that $\gamma_1, \gamma_2\in H_1(X,\bZ)$ are such that $\gamma_1, \gamma_2, M^*\gamma_1, M^*\gamma_2$ are a basis, where $M^*$ denotes the dual linear endomorphism of $H_1(X,\bZ)$. Then the linear equations are 
$$\int_{M^*\gamma_i}\omega = \sqrt{D}\int_{\gamma_i}\omega, \quad\text{for}\quad i=1,2.$$

\bold{Examples.} Let us now show the these loci are nonempty in $\cH(2)$, by giving explicit examples. A similar argument could show that these loci are nonempty in $\cH(1,1)$. (But in $\cH(1,1)$ there is a softer argument as well, which we will not present.)

\begin{prop}
The surfaces indicated in figure \ref{F:H2eigenform} are in the loci described in Proposition \ref{P:RMwoRM}. (But the $D$ in the proposition might not be the same as the $D$ in the figure.)
\begin{figure}[h]
\centering
\includegraphics[scale=0.45]{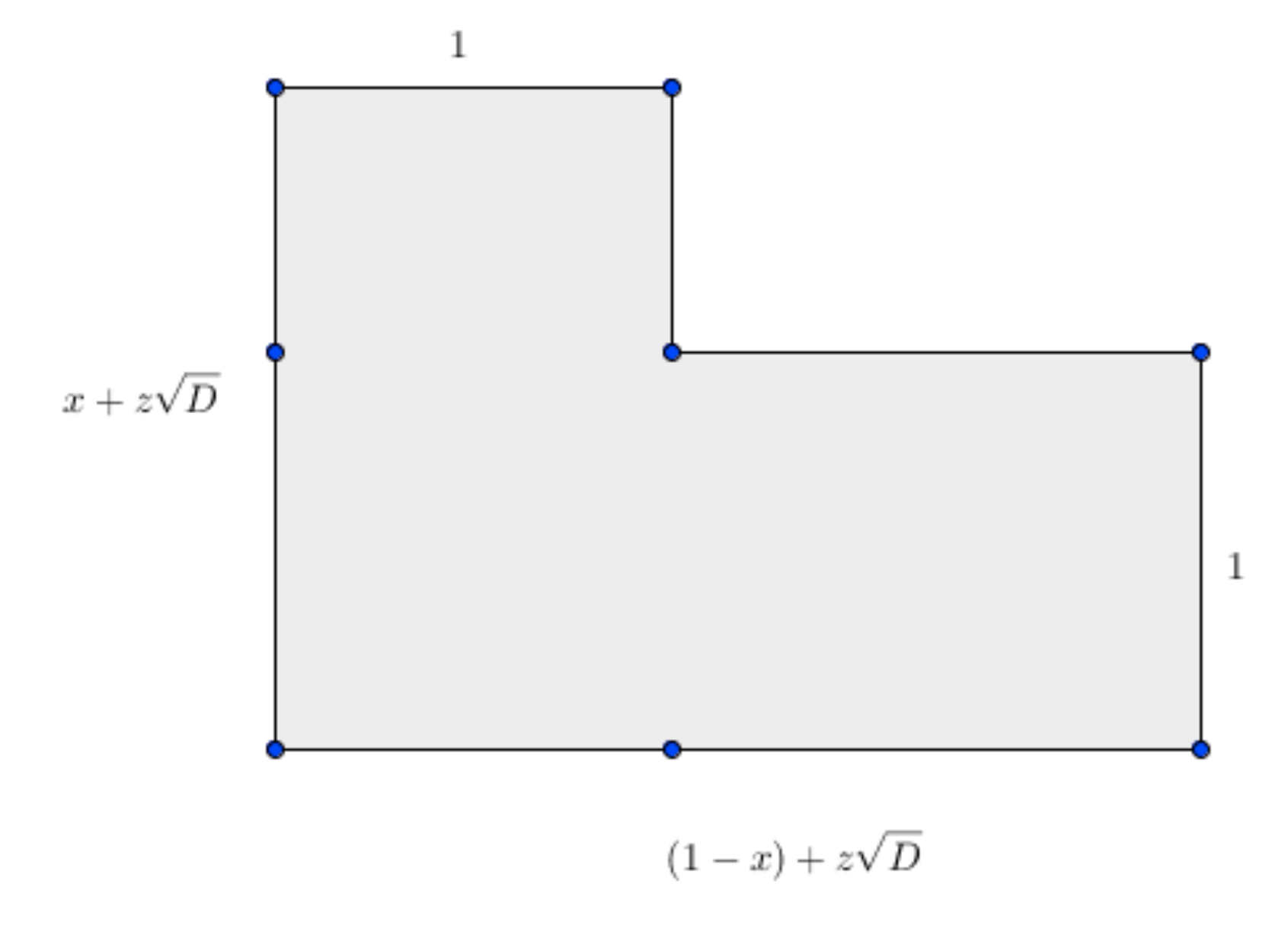}
\caption{Opposite sides are identified, and all edges are vertical or horizontal. The numbers indicate signed real length in the direction indicated, and the parameters $x,z$ are rational.}
\label{F:H2eigenform}
\end{figure}
\end{prop}

\begin{lem}
For the translation surface $(X,\omega)$ in figure \ref{F:H2eigenform}, $$\span(\Re(\omega), \Im(\omega))$$ is symplectically orthogonal to the Galois conjugate subspace of $H^1(X,\omega)$. 
\end{lem}
\begin{proof}
Since the periods of $\omega$ are in $\bQ[\sqrt{D}, i]$, we can define a cohomology class $\omega'$ in $H^1(X, \bQ[\sqrt{D}, i])$ that is Galois conjugate to $\omega$. Note that $\omega'$ is not expected to be represented by a holomorphic 1-form. It can be described more concretely via the isomorphism 
$$H^1(X, \bQ[\sqrt{D}, i])=\Hom(H_1(X,\bZ), \bQ[\sqrt{D}, i]).$$
From this point of view, $\omega'$ is the composition of $\omega$ with the field endomorphism of $\bQ[\sqrt{D}, i]$ that fixes $i$ and sends $\sqrt{D}$ to $-\sqrt{D}$. 

\begin{figure}[h]
\centering
\includegraphics[scale=0.35]{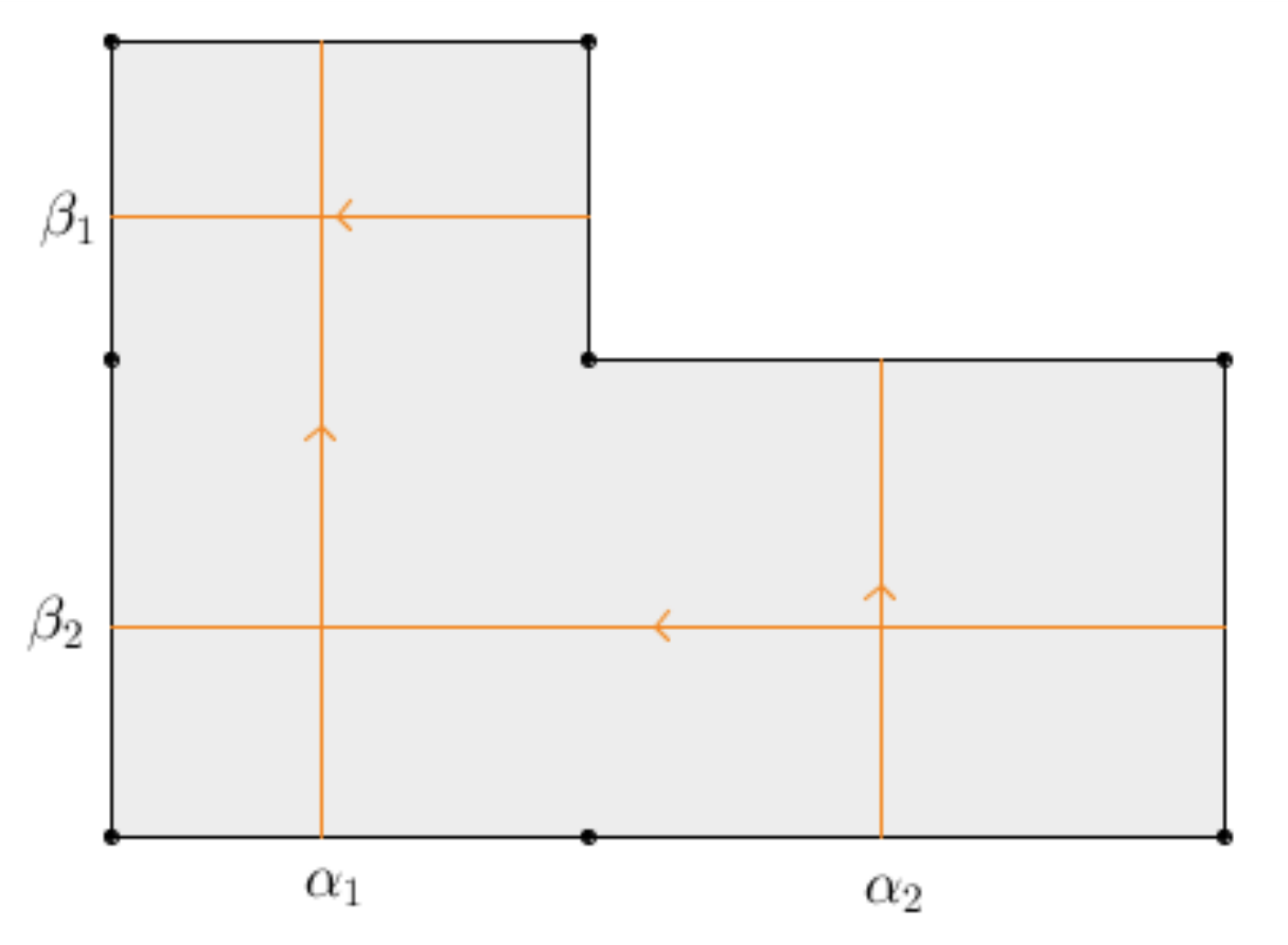}
\caption{A basis of homology is indicated. }
\label{F:EllBasis}
\end{figure}
In figure \ref{F:EllBasis}, note that $\alpha_1, \beta_1, \alpha_2, \beta_2-\beta_1$ is symplectic basis of homology. So, if $\phi,\psi \in H^1(X, \bR)=\Hom(H_1(X,\bR), \bR)$, we can compute their symplectic pairing as 
\begin{eqnarray*}
\langle \phi, \psi \rangle &=& \phi(\alpha_1)\psi(\beta_1)-\phi(\beta_1)\psi(\alpha_1)+
\\&&\phi(\alpha_2)\psi(\beta_2-\beta_1)-\phi(\beta_2-\beta_1)\psi(\alpha_2).
\end{eqnarray*}
Suppose that we set $A_j=\frac1i\int_{\alpha_j}\omega$ and $B_j=\int_{\beta_j} \omega$. We are assuming that $B_1=-1$ and $A_2=1$. Then we may compute
\begin{eqnarray*}
\langle \Im(\omega), \Re(\omega') \rangle &=& A_1B_1'+A_2(B_2'-B_1')  
\\&=&
-(A_1-B_2'-1).
\end{eqnarray*}
If $A_1=x+z\sqrt{D}$ and $B_2=y+w\sqrt{D}$, this quantity is zero if and only if  $x-y=1$ and $w=-z$. 

Both  
$\langle \Re(\omega), \Re(\omega') \rangle$ and $\langle \Im(\omega), \Im(\omega') \rangle$ are easily seen to be automatically zero, and 
$\langle \Re(\omega), \Im(\omega') \rangle $ is the Galois conjugate of $-\langle \Im(\omega), \Re(\omega') \rangle$. Hence these conditions $x-y=1$ and $w=-z$ are equivalent to the orthogonally of $\span(\Re(\omega), \Im(\omega))$ and $\span(\Re(\omega'), \Im(\omega'))$. 
\end{proof}

\begin{proof}[Proof of proposition.]
We define $M_0$ to be the endomorphism of $H^1(X, \bR)$ that acts by $\sqrt{D}$ on $\span(\Re(\omega), \Im(\omega))$ and by $-\sqrt{D}$ on the Galois conjugate subspace $\span(\Re(\omega'), \Im(\omega'))$. This matrix is rational, and it is self-adjoint if and only if $\span(\Re(\omega), \Im(\omega))$ and $\span(\Re(\omega'), \Im(\omega'))$ are symplectically orthogonal. 

Since $M_0$ is a rational matrix, for some $m>0$ we get that $M=m^2M_0$ is integral. Thus $(X,\omega)$ is in the locus constructed in the Proposition \ref{P:RMwoRM}, but the $D$ value in the proposition is $m^2D$ here.  
\end{proof}

\bold{Connection to real multiplication.} Recall that the complex vector space $H^{1,0}(X)$ is canonically isomorphic as a real vector to $H^1(X,\bR)$, via the map $\omega\mapsto\Re(\omega)$. The natural symplectic form on $H^1(X,\bR)$ pulls back to the pairing
$$\langle \omega_1, \omega_2\rangle =\frac12\Re\left( \int\omega_1 \overline{\omega_2}\right)$$ 
on $H^{1,0}(X)$. This symplectic form on $H^{1,0}(X)$ is compatible with the complex structure on $H^{1,0}(X)$, in that $\langle i\omega_1, i\omega_2\rangle=\langle \omega_1, \omega_2\rangle$ and $-\langle \omega, i\omega\rangle >0$ for all $\omega\neq 0$.

We will let $H^{1,0}(X)^*$ denote the dual of the complex vector space $H^{1,0}(X)$. As a real vector space, $H^{1,0}(X)^*$ is canonically isomorphic to $H_1(X,\bR)$, via the usual integration pairing between homology classes and Abelian differentials.  The space $H^{1,0}(X)^*$ inherits a dual compatible symplectic pairing, which we will also denote $\langle \cdot, \cdot\rangle$. 

\begin{defn}
The \emph{Jacobian} $\Jac(X)$ of a Riemann surface $X$ is the complex torus $H^{1,0}(X)^*/ H_1(X, \bZ)$, together with the data of the symplectic pairing $\langle \cdot, \cdot\rangle$ on its  tangent space to the zero. (This tangent space is $H^{1,0}(X)^*$.) An \emph{endomorphism} of $\Jac(X)$ is an endomorphism of the complex torus. An endomorphism is called self-adjoint if the induced endomorphism on the tangent space to the identity is self-adjoint with respect to the symplectic form. 
\end{defn}

Thus, endomorphisms of $\Jac(X)$ can be thought of either as a complex linear endomorphisms of $H^{1,0}(X)^*$ that preserves $H_1(X, \bZ)$, or as a linear endomorphisms of $H_1(X, \bZ)$ whose real linear extension to $H^{1,0}(X)^*$ happens to be complex linear. From our point of view, the second perspective is more enlightening, and the requirement of complex linearity is the deepest part of the definition. This is because the complex structure on $H^{1,0}(X)^*$ is determined by how $H^{1,0}(X)$ sits in $H^1(X,\bC)$, i.e., the Hodge decomposition. The Hodge decomposition varies as the complex structure on $X$ varies, in a somewhat mysterious way. The complex linearity restriction is the only part of the data of an endomorphism of $\Jac(X)$ that depends on the complex structure on $X$ (the rest could be defined for a topological surface instead of a Riemann surface).

Recall that a \emph{totally real field} is a finite field extension of $\bQ$, all of whose field embeddings into $\bC$ have image in $\bR$. Every real quadratic field is totally real, but there are cubic real fields that are not totally real, for example $\bQ[2^{\frac13}]$. An \emph{order} in a number field is just a finite index subring of the ring of integers. The key example to keep in mind is that $\bZ[\sqrt{D}]$ is an order in $\bQ[\sqrt{D}]$.  

\begin{defn}\label{D:RM}
The Jacobian $\Jac(X)$ of a Riemann surface of genus $g$ is said to have \emph{real multiplication} by a totally real number field $\bk$ of degree $g$ if there is some order $\cO\subset \bk$  that acts on $\Jac(X)$ by self-adjoint endomorphisms. An Abelian differential $\omega\in H^{1,0}(X)$ is said to be an \emph{eigenform} for this action if it is an eigenvector for the induced action of $\cO$ on the cotangent space of $\Jac(X)$ at 0. (This cotangent space is $H^{1,0}(X)$.)
\end{defn}

Genus two is special for real multiplication because of the following. 

\begin{lem}
Fix a compatible symplectic structure on $\bC^2$, and let $M:\bC^2\to \bC^2$ be a self-adjoint real linear endomorphism. If $M$ preserves a complex line $L\in \bC^2$, then $M$ is complex linear.  
\end{lem}

\begin{proof}
First note, that the only self-adjoint real linear endomorphisms of $\bR^2$ are scalars times identity. (This can be verified in coordinates, for 2 by 2 matrix and the standard symplectic form on $\bR^2$.) 

Since the symplectic form is compatible, $L^\perp$ is a complex line. $M$ leaves invariant the two complex lines $L$ and $L^\perp$, and acts as a scalar on each, so it must be complex linear. 
\end{proof}

It follows from this that 

\begin{thm}[McMullen]
The  loci in $\cH(2)$ and $\cH(1,1)$ defined in Proposition  \ref{P:RMwoRM} in fact parameterize $(X,\omega)$ where $\Jac(X)$ admits real multiplication  by $\bQ[\sqrt{D}]$ with $\omega$ as an eigenform. 
\end{thm}

\begin{proof}
From the definition of the eigenform loci we get a self-adjoint action of $\bZ[\sqrt{D}]$ on $H^1(X,\bZ)$, where $\sqrt{D}$ acts by $M$. This gives a dual action on $H_1(X,\bR)$, which preserves $H_1(X,\bZ)$. It suffices to show that this action is complex linear. This follows from the lemma, and the observation that the complex line consisting of the annihilator of $\span(\Re(\omega), \Im(\omega))$ is preserved by the action.
\end{proof}

\begin{rem}
For fixed $D$, the locus of $(X,\omega)$ where $\Jac(X)$ admits real multiplication  by $\bQ[\sqrt{D}]$ with $\omega$ as an eigenform in fact has infinitely many connected components. This is related to the fact that there are infinitely many orders in $\bQ[\sqrt{D}]$. If one fixes the maximal order that acts on $\Jac(X)$ the situation is greatly ameliorated, for example the locus in $\cH(1,1)$ is connected and is closely related to a Hilbert modular surface. McMullen's original treatment \cite{Mc} keep track of the maximal order, because it is necessary to study connected components and give the relation to Hilbert modular surfaces. See also \cite{KM}.
\end{rem}

\begin{rem}
In McMullen's original work \cite{Mc}, the study of real multiplication arose naturally from flat geometry and low dimensional topology in the following way. As we will discuss later in these notes, every $(X,\omega)$ in a two complex dimensional affine invariant submanifold must possess many affine symmetries. McMullen showed that in genus 2, certain affine symmetries naturally give rise to real multiplication on $\Jac(X)$, and that $\omega$ is an eigenform. 
\end{rem}

\begin{rem}
It is important to note  the proof of Proposition \ref{P:RMwoRM} doesn't work in higher genus. This is because the $M_n$ are not guaranteed to converge (they might diverge). (Here we fix a generator for an order in a totally real number field, and $M_n$ is the action of this generator.)

Given that the $M_n$ are automatically complex linear in genus 2, it is natural to impose this condition in higher genus, and try to see if loci of eigenforms as defined in definition \ref{D:RM} give affine invariant submanifolds. With the complex linearity, it turns out the $M_n$ must converge, which at least lets you show the locus is closed. 


However, even with complex linearity the end result isn't true, because even if $M$ is complex linear at $(X,\omega)$, there is no reason for this to be true at $(X',\omega')$. So in higher genera eigenforms are locally contained in nice linear spaces, but the real multiplication can vanish as you move from an eigenform $(X,\omega)$ to a nearby translation surface $(X',\omega')$ whose periods satisfy the same linear equations. 
\end{rem}

\subsection{Torsion, real multiplication, and algebraicity} There is however a still a very strong connection between endomorphisms of Jacobians and affine invariant submanifolds in higher genus, as discovered by M\"oller in the case that the affine invariant submanifold has  complex dimension 2. 

\begin{defn}
Given $(X,\omega)$, let $V\subset H^1(X,\bQ)$ be the smallest subspace such that $V\otimes \bC$ contains $\omega$, and such that $V\otimes \bC=V^{1,0}\oplus V^{0,1}$, where $$V^{1,0}=(V\otimes \bC)\cap H^{1,0}\quad\text{and}\quad V^{0,1}=\overline{V^{1,0}}.$$  Set $V^*_\bZ=\{\phi\in V^*: \phi(V\cap H^1(X,\bZ))\subset \bZ\}$. Define $$\Jac(X,\omega)=(V^{1,0})^*/V^*_\bZ,$$ together with the data of the symplectic form on  $V^*$. (This object $\Jac(X,\omega)$ does not have a standard name, and this notation is not standard.)
\end{defn}

$V$ inherits a symplectic form by restriction from $H^1(X,\bQ)$. The restriction is automatically symplectic (i.e., nondegenerate) because of the condition $V\otimes \bC=V^{1,0}\oplus V^{0,1}$. The symplectic form on $V^*$ is the dual symplectic form. 

$\Jac(X,\omega)$ is a factor of $\Jac(X)$ up to isogeny, and it the ``smallest factor containing $\omega$". 

\begin{defn}\label{D:torsion}
Let $p, q$ be two points of $(X,\omega)$. We say that $p-q$ is \emph{torsion} in $\Jac(X,\omega)$ if, for any relative homology class $\gamma_{p,q}$ of a curve from $p$ to $q$, there is a $\gamma\in H_1(X,\bQ)$ such that for all $\omega'\in V^{1,0}$ (including the most important one $\omega'=\omega$), 
$$\int_{\gamma_{p,q}}\omega' = \int_\gamma \omega'.$$
\end{defn}

If $\Jac(X,\omega)=\Jac(X)$, this is equivalent to $p-q$ being torsion in the group $\Jac(X)$. For every affine invariant submanifold $\cM$ of complex dimension two, there is a natural algebraic extension of $\bQ$ called the trace field, which will be defined in the final section. 
The definition of real multiplication on $\Jac(X,\omega)$ is exactly analogous to that for $\Jac(X)$, except the degree of the field is required to be equal to the complex dimension of $\Jac(X,\omega)$, which is sometimes less than $g$, and the order is required to act on $\Jac(X,\omega)$ (instead of $\Jac(X)$). 

\begin{thm}[M\"oller \cite{M, M2}]
For every affine invariant submanifold $\cM$ of complex dimension two, and every $(X,\omega)\in \cM$, $\Jac(X,\omega)$ has real multiplication by an order in the trace field with $\omega$ as an eigenform, and furthermore if $p$ and $q$ are zeros of $\omega$, then $p-q$ is torsion in $\Jac(X,\omega)$.
\end{thm}

This result is actually two deep theorems, and their importance for the field is very great. One might expect that, given that an affine invariant submanifold is defined in terms of $(X,\omega)$, the other holomorphic 1-forms on $X$ would not be special, but this result says that they are (since the definition of torsion involves all or many $\omega'$, and real multiplication produces other eigenforms for the Galois conjugate eigenvalues). 

The converse of M\"oller's result is true, and is much easier than the result itself. 

\begin{prop}[Wright]\label{P:Mconverse}
Let $\cM$ be a 2-dimensional submanifold of a stratum (not assumed to be linear), and suppose that for every $(X,\omega)\in \cM$, $\Jac(X,\omega)$ has real multiplication by the trace field with $\omega$ as an eigenform, and furthermore if $p$ and $q$ are zeros of $\omega$, then $p-q$ is torsion in $\Jac(X,\omega)$. Then $\cM$ is an affine invariant submanifold.
\end{prop}

\begin{proof}
We must show that $\cM$ is locally linear. For notational simplicity, we will assume $\Jac(X,\omega)=\Jac(X)$. 

There are only countably many totally real number fields of degree $g$, and only countably many actions of each on $H^1(X,\bQ)$. Since the locus of eigenforms for each is closed, we may assume that the totally real number field is constant, and the action on $H^1(X,\bQ)$ is locally constant. (Again, this is a simple connectivity argument: The connected space $\cM$ is covered by disjoint closed sets, so there can only be one.) 

Similarly, we can assume that for each pair of zeros $p,q$ of $\omega$, the rational homology class $\gamma$ in definition \ref{D:torsion} is locally constant. 

We will show that the span of $\Re(\omega)$ and $\Im(\omega)$ does not change in $H^1(X,\Sigma, \bC)$. This span gives the linear subspace that locally defines $\cM$. 

The real multiplication condition gives that the image of this span is locally constant in absolute homology (it is an eigenspace), and the torsion condition gives that each relative period $\int_{\gamma_{p,q}}\omega$ is equal to an absolute period $\int_\gamma \omega$, and hence the relative periods are determined by absolute periods. 

That concludes the proof, but in closing we will write down the linear equations explicitly at $(X,\omega)\in \cM$, in the case where $\Jac(X,\omega)=\Jac(X)$. Say the number field is $\bk$ and has $\bQ$-basis $r_1, \ldots, r_g$. Pick two absolute homology classes $\gamma_1, \gamma_2$ so that  $$\rho(r_i) \gamma_j,\quad i=1, \ldots, g,\quad j=1,2,$$ are a basis for $H_1(X,\bQ)$. Say the zeros of $\omega$ are $p_1, \ldots, p_s$, and let $\alpha_i$ be a relative cycle from $p_i$ to $p_s$, for $i=1, \ldots, s-1$. For each $i$, let $\alpha_i'\in H_1(X,\bQ)$ be given from $\alpha_i$ by definition \ref{D:torsion}. Then the equations are
$$\int_{\rho(r_i)\gamma_j}\omega = r_i \int_{\gamma_j}\omega, \quad i=1, \ldots, g,\quad j=1,2, \quad$$and$$\int_{\alpha_i}\omega = \int_{\alpha_i'} \omega,\quad i=1, \ldots, s-1.$$
\end{proof}

Recently, Simion Filip has generalized M\"oller's result to  affine invariant submanifolds of any dimension. As in the previous proposition, this gives enough algebro-geometric conditions to characterize affine invariant submanifolds, and so Filip is able to conclude the following \cite{Fi1, Fi2}. 

\begin{thm}[Filip]
All affine invariant submanifolds are quasi-projective varieties. 
\end{thm}

Filip's proof crucially uses dynamics, and no other proof is known. 

\section{The action of $\G $}

This section  will begin to set the stage for the following driving theme in the study of translation surfaces.  
\vspace{0.10in}
{\quote\emph{The behavior of certain dynamical systems is fundamentally linked to the structure of affine invariant submanifolds.}}

\vspace{0.10in}
The dynamical systems involved are the action of $\G$ on each stratum, which is the topic of this section, and the straight line flow on each individual translation surface, which is the topic of the next section. It is very important to note that the connections go in both directions: Dynamical information often powers structural results on affine invariant submanifolds, and in the opposite direction results about the structure of linear manifolds are crucial in the study of the dynamical problems.

\subsection{Definitions and basic properties}
In this section, it is most helpful to think of a translation surface using the third definition (polygons). Let $\G$ be the group of two by two matrices with positive determinant. 

There is an action of $\G $ on each stratum $\cH$ induced from the linear action of $\G$ on $\bR^2$. If $g\in \G $, and $(X,\omega)$ is a translation surface given as a collection of polygons, then $g(X,\omega)$ is the translation surface given by the collection of polygons obtained by acting linearly by $g$ on the polygons defining $(X,\omega)$. 

\begin{figure}[h!]
\centering
\includegraphics[scale=0.41]{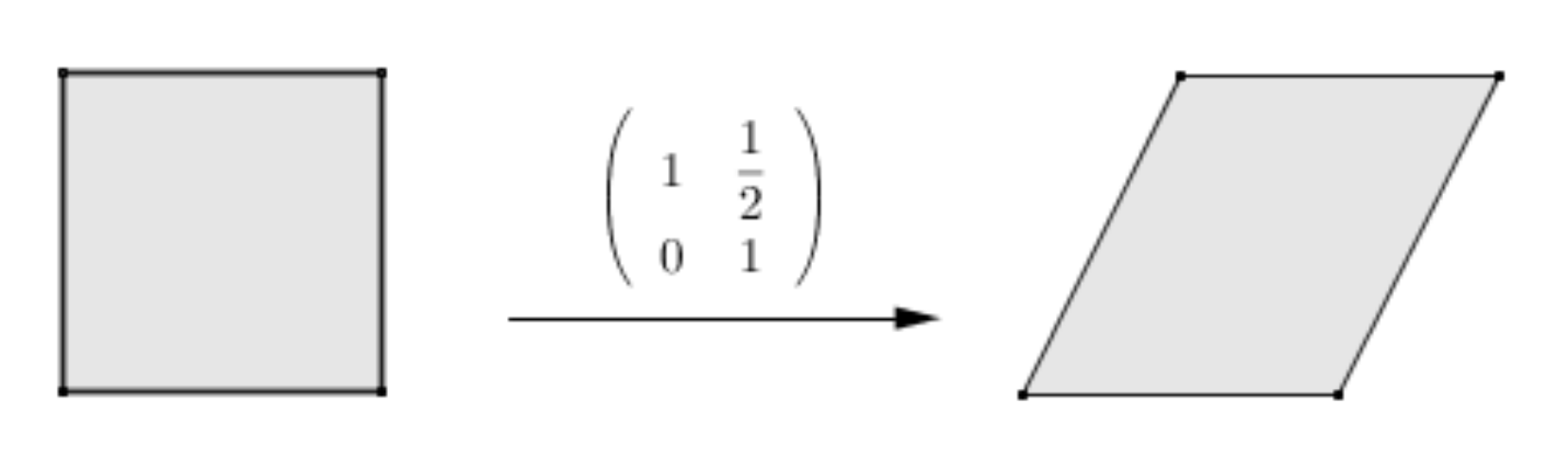}
\end{figure}

Naively, one might thing that $g(X,\omega)$ is always very different from $(X,\omega)$ when $g$ is a large matrix, because $g$ distorts any polygon a large amount. But because of the cut and paste equivalence, this is not the case. For example,

\begin{prop}
The stabilizer of $(\bC/\bZ[i], dz)$ is $\GL(2,\bZ)=SL(2,\bZ)$. The stabilizer of any square-tiled surface is a finite index subgroup of $\GL(2,\bZ)$.
\end{prop}

\begin{proof}
Note that $g(\bC/\bZ[i], dz) = (\bC/g(\bZ[i]), dz)$. This gives that the stabilizer of $(\bC/\bZ[i], dz)$ is exactly the matrices $g\in \G$ preserving $\bZ[i]\subset \bC$. Hence the stabilizer of $(\bC/\bZ[i], dz)$ is $SL(2,\bZ)$.

In general, the stabilizer preserves the periods of a surface. For a square-tiled surface, the periods are $\bZ[i]$, so for any square-tiled surface the stabilizer is a subgroup of $SL(2,\bZ)$. 

Now, suppose $(X,\omega)$ is a square-tiled surface. For any $g\in SL(2,\bZ)$, $g(X,\omega)$  is a square-tiled surface with the same number of squares. Hence, the stabilizer of $(X,\omega)$ is finite index in $SL(2,\bZ)$. 
\end{proof}

More basic properties are

\begin{prop}
The $SL(2,\bR)$-orbit of every translation surface is unbounded. The stabilizer of every translation surface is discrete but never cocompact in $SL(2,\bR)$.
\end{prop} 

\begin{proof}
If $g\in \G $ is close enough to the identity, then $g(X,\omega)$ and $(X,\omega)$ are in the same coordinate chart. The coordinates of $g(X,\omega)$ are obtained from those of $(X,\omega)$ by acting linearly on $\bC=\bR^2$. This shows that for $g$ sufficiently small enough and not the identity, $g(X,\omega)\neq (X,\omega)$, because they have different coordinates. Hence the stabilizer is discrete. 

Set 
$$r_\theta=\left(\begin{array}{cc} \cos(\theta)&-\sin(\theta)\\\sin(\theta)&\cos(\theta)\end{array}\right)\quad\text{and}\quad
g_t=\left(\begin{array}{cc} e^t&0\\0&e^{-t}\end{array}\right).$$
For every $(X,\omega)$, there is some $\theta$ such that $r_\theta(X,\omega)$ has a vertical saddle connection. (Pick any saddle connection, and rotate so that it becomes vertical.) On $g_tr_\theta(X,\omega)$ this saddle connection is $e^{-t}$ times as long, so as $t\to\infty$, we see that  $g_tr_\theta(X,\omega)$ has shorter and shorter saddle connections and hence diverges to infinity in the stratum. 

Let $\Gamma$ be the stabilizer of $(X,\omega)$, and suppose $(X,\omega)$ lies in the stratum $\cH$. Consider the natural orbit map $$SL(2,\bR) /\Gamma\to \cH, \quad\quad [g]\mapsto g(X,\omega).$$ By definition, $\Gamma$ is cocompact if $SL(2,\bR) /\Gamma$ is compact. If $SL(2,\bR) /\Gamma$ were compact, then its image under this map would be compact also. However, the image is just the $SL(2,\bR)$-orbit of $(X,\omega)$, which must be unbounded and hence cannot be compact. 
\end{proof}

\begin{defn}
A \emph{cylinder} on a translation surface is an isometrically embedded copy of a Euclidean cylinder $(\bR/c\bZ)\times (0,h)$ whose boundary is a union of saddle connections. The number $c$ is the \emph{circumference}, and the number $h$ is the \emph{height} of the cylinder. The ratio $h/c$ is the \emph{modulus} of the cylinder. The \emph{direction} of the cylinder is the direction of its boundary saddle connections (so directions are formally elements of $\bR P^1$). A translation surface $(X,\omega)$ is called \emph{periodic} in some direction $(X,\omega)$ is the union of the cylinders in that direction together with their boundaries. 
\end{defn}

\begin{figure}[h!]
\centering
\includegraphics[scale=0.30]{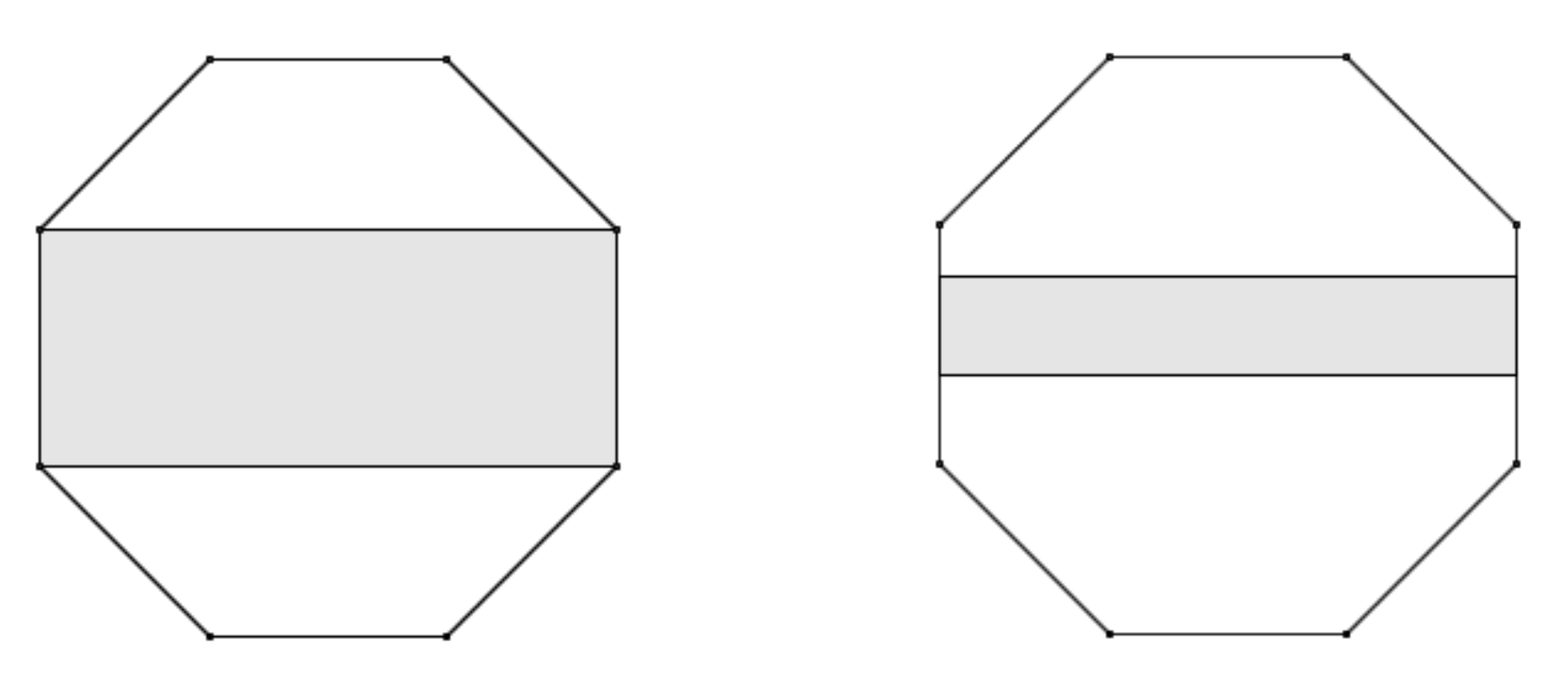}
\caption{On the left octagon (with opposite sides identified), both the shaded part and its complement are horizontal cylinders. On the right octagon, the shaded rectangle is \emph{not} a cylinder according to our definition, because its boundary is not a union of saddle connections. The effect of requiring the boundary to consist of saddle connections is equivalent to requiring the cylinders to be ``maximal", unlike this example on the right, whose height could be increased. The regular octagon is horizontally periodic, since it is the union of two horizontal cylinders (and their boundaries).}
\end{figure}

\begin{prop}\label{P:para}
A translation surface contains a matrix of the form 
$$\left(\begin{array}{cc}1&t\\0&1\end{array}\right)$$
in its stabilizer if and only if the surface is horizontally periodic, and all reciprocals of moduli of horizontal cylinders are integer multiples of $t$. 
\end{prop}

\begin{proof}
\begin{figure}[h!]
\centering
\includegraphics[scale=0.5]{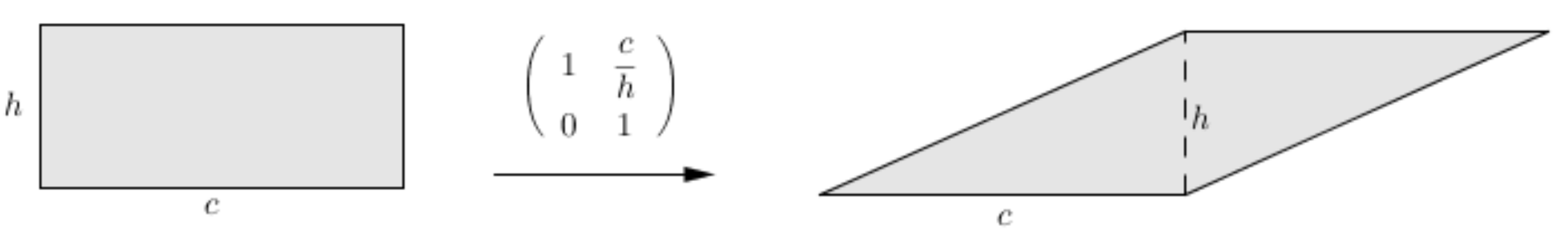}
\caption{Each individual cylinder is stabilized by a parabolic matrix. In this picture, opposite edges are identified, except for the horizontal edges, which are the upper and lower boundary of the cylinder. The cylinder on the right can be cut along the dotted line and re-glued to give the cylinder on the left.}
\label{F:Parabolic}
\end{figure}

Assume the surface is horizontally periodic, and all reciprocals of moduli of horizontal cylinders are integer multiples of $t$. As is illustrated in Figure \ref{F:Parabolic}, each horizontal cylinder of modulus $m$ is stabilized by the matrix 
$$\left(\begin{array}{cc}1&m^{-1}\\0&1\end{array}\right).$$
The result follows. 

For the converse, see \cite{MT}. 
\end{proof}

Note that the previous  proposition applies to any direction (not just horizontal) by first rotating the surface. It is stated for the horizontal direction only for notational simplicity.

\subsection{Closed orbits and orbit closures}\label{S:closed}

\begin{prop}
Affine invariant submanifolds are $\G $ invariant.
\end{prop}

\begin{proof}
It suffices to show that for each $(X,\omega)$  in the affine invariant submanifold $\cM$ there is a small neighbourhood $U$ of the identity in $\G $ such that for all $g\in U$ we have $ g(X,\omega)\in\cM$. This gives that the set of $g$ for which $g(X,\omega)\in \cM$ is open in $\G $. Since $\cM$ is closed and the action is continuous, this set is also closed. Since $\G $ is connected,  any nonempty subset that is both open and closed must be everything.

If $g$ is small enough, both $g(X,\omega)$ and $(X,\omega)$ are in the same coordinate chart. The coordinates of $g(X,\omega)$ are obtained by letting $g$ act linearly on the real and imaginary parts of each coordinate of $(X,\omega)$, using the isomorphism $\bC\cong \bR^2$. 

For example, $g_t$ scales the real part of the coordinates by $e^t$ and the imaginary part of the coordinates by $e^{-t}$. And $u_t$ adds $t$ times the imaginary part of the coordinates to the real coordinates.

Since $\cM$ is defined by linear equations with \emph{real} coefficients and constant term 0, both the real and imaginary parts of the coordinates also satisfy the linear equations, as well as any complex linear combination of them. 
\end{proof}

That a converse is true is a recently established very deep fact, due to Eskin-Mirzakhani-Mohammadi \cite{EM, EMM}. The proof is vastly beyond the scope of these notes. When we say ``closed" below, we continue to refer to the analytic topology on a stratum. 

\begin{thm}[Eskin-Mirzakhani-Mohammadi]
Any closed  $\G $ invariant set is a finite union of affine invariant submanifolds. In particular, every orbit closure is an affine invariant submanifold. 
\end{thm}

This theorem is false if $\G $ is replaced with the diagonal subgroup: there are closed sets invariant under the diagonal subgroup that are locally homeomorphic to a Cantor set cross $\bR$. Determining to what extent the theorem holds for the unipotent subgroup 
$$\left(\begin{array}{cc} 1&t\\0&1 \end{array}\right)$$
is a major open problem. 

The context for these theorems comes from homogenous space dynamics, where Ratner's Theorems give that orbit closures of a unipotent flow on a homogenous space must be sub-homogenous spaces. 

Because of the work of Eskin-Mirzakhani-Mohammadi (and a converse that will we discuss below), the term ``affine invariant submanifold" is synonymous with ``$\G$-orbit closure", usually abbreviated ``orbit closure". Sometimes we will consider orbit closures for actions of subgroups of $\G$, but in these cases we will make this clear by specifying the subgroup. The default is $\G$ (or, for some other people, $SL(2,\bR)$).

\bold{Closed orbits.}
Let us now turn to the case of $2$-dimensional affine invariant submanifolds $\cM$. In this case  $\G$ acts transitively on $\cM$. (Indeed, the real dimension of $\G$ is equal to that of $\cM$, and it is easily checked that the $\G$-orbit of any point is open in $\cM$. The different $\G$-orbits in $\cM$ are open disjoint sets, so since $\cM$ is connected there must only be one. Hence $\cM$ is the $\G$-orbit of any translation surface in $\cM$.)

Thus ``closed $\G$-orbit" is synonymous with ``2-dimensional affine invariant submanifold." 

\begin{rem}
We have already given a number of examples of 2-dimensional affine invariant submanifolds, namely the eigenform loci in $\cH(2)$, and spaces of branched covers of genus 1 translation surfaces branched over exactly 1 point. 
\end{rem}

The following result has more than one proof \cite{V5, SW2}, 
but all known proofs follow the same basic outline and use dynamics in a nontrivial way. 

\begin{thm}[Smillie]
Suppose that $(X,\omega)\in \cH$ has closed orbit and stabilizer $\Gamma$. Then the orbit map 
 $$\G /\Gamma\to \cH, \quad\quad [g]\mapsto g(X,\omega)$$
 is a homeomorphism, and $\Gamma$ is a lattice in $SL(2,\bR)$.  
\end{thm}

The first conclusion is entirely expected and is standard (but  technical to prove), and the second is deep and fundamentally important. 

\begin{rem}
All square tiled surfaces are part of 2-dimensional affine invariant submanifolds and hence have closed orbit. The stabilizer of a square-tiled surface is a finite index subgroup of $SL(2,\bZ)$.  

The eigenforms illustrated in figure \ref{F:H2eigenform} must, by Smillie's theorem, have lattice stabilizer. However, it is very hard to write down this stabilizer, even in specific examples. The orbifold type of the stabilizer was computed in  \cite{McM:spin, Ba, Mu:orb}, and an algorithm for computing the stabilizer was given in \cite{Mu:alg}.
\end{rem}

Finally, we note that if $\cM$ is a 2-dimensional orbit closure, then the projection to the moduli space of Riemann surfaces (via $(X,\omega)\mapsto X$) is 1-dimensional. One complex dimension is lost, since $(X,\omega)$ and $(X,r \omega)$ map to the same point, for any $r\in\bC$. A corollary of Smillie's theorem (and also Filip's theorem, obtained much later) is that this projection of $\cM$ is in fact an algebraic curve. 

\begin{prop}
The projection of closed orbit to the moduli space of Riemann surfaces is an algebraic curve, which is isometrically immersed with respect to the Teichm\"uller metric. 
\end{prop}

Isometrically immersed curves in the moduli space of Riemann surfaces are called \emph{Teichm\"uller curves}. Up to a ``double covering" issue relating quadratic differentials to Abelian differentials, all Teichm\"uller curves are projections of closed orbits. 

Royden has shown that the Kobayashi metric on the moduli space of Riemann surfaces is equal to the Teichm\"uller metric. Using this, McMullen has shown that Teichm\"uller curves are rigid \cite{Mc3}. The study of Teichm\"uller curves is a fascinating area at the intersection of dynamics and algebraic geometry. 

See \cite{W2} for a list of known Teichm\"uller curves, and see \cite{Mc4, M3, BaM, MW} for some finiteness results.

\bold{Stable and unstable manifolds for $g_t$.} We will now give a flavor of the dynamics of the $g_t$ action. We will not return to this explicitly in these notes, but the ideas we present now underlie most of the proofs that have been omitted from these notes. 

Say the period coordinates of $(X,\omega)$ are $v_j=x_j+i y_j$ for $j=1, \ldots, n$. We can then think of $(X,\omega)$ in coordinates as a $2$ by $n$ matrix whose first row gives the real parts of period coordinates, and whose second row gives the imaginary parts.
$$\left(\begin{array}{cccc} x_1&x_2&\cdots &x_n\\y_1&y_2&\cdots &y_n \end{array}\right)$$
The advantage of writing the coordinates this way is that any $g\in \G$ close to the identity (so $g(X,\omega)$ stays in the same chart) acts by left multiplication. In particular, for small $t$ we have that $g_t(X,\omega)$ is the matrix product
$$\left(\begin{array}{cc} e^t&0\\0&e^{-t}\end{array}\right) \left(\begin{array}{cccc} x_1&x_2&\cdots &x_n\\y_1&y_2&\cdots &y_n \end{array}\right).$$
However, for large $t$, we expect $g_t(X,\omega)$ will have left the coordinate chart. When it enters a different coordinate chart, a change of basis matrix can be used to compute the new coordinates from the old. This matrix is a $n$ by $n$ invertible integer matrix, which we will call $A_t(X,\omega)$. (This is an imprecise definition of $A_t(X,\omega)$, which of course depends on  choices of coordinates, etc.)

We get that the coordinates of $g_t(X,\omega)$ are 
$$\left(\begin{array}{cc} e^t&0\\0&e^{-t}\end{array}\right) \left(\begin{array}{cccc} x_1&x_2&\cdots &x_n\\y_1&y_2&\cdots &y_n \end{array}\right)A_t(X,\omega).$$
The matrix $A_t(X,\omega)$ is called the \emph{Kontsevich-Zorich cocycle}. It is a cocycle in the dynamical systems sense, which simply means 
$$A_{t+s}(X,\omega) = A_t(g_s(X,\omega)) A_s(X,\omega).$$
The Kontsevich-Zorich cocycle is the complicated part of the dynamics of $g_t$. However, its effect is usually beat out by the effect of the $(e^t)$'s on the left. In particular, 

\begin{thm}[Masur, Veech, Forni \cite{Ma2,V3, Fex1}]
Fix an affine invariant submanifold $\cM$. 
For almost every $(X,\omega)\in \cM$, and every  $(X', \omega')$ in the same coordinate chart with the same real parts of period coordinates,  the distance between $g_t(X,\omega)$ and $g_t(X', \omega')$ goes to zero as $t\to\infty$. 
\end{thm}

Without the interference of the Kontsevich-Zorich cocycle, this would simply be the obvious fact that 
$$\left(\begin{array}{cc} e^t&0\\0&e^{-t}\end{array}\right) \left(\begin{array}{cccc} 0&0&\cdots &0\\y_1-y_1'&y_2-y_2'&\cdots &y_n-y_n' \end{array}\right)\to 0.$$
The theorem says that even with the Kontsevich-Zorich cocycle added in on the left  this still holds. 

Overall, $g_t$ expands the real parts of period coordinates exponentially, and contracts the imaginary parts exponentially. But the reader should be warned that there are technical complications arising from the fact that strata are not compact. 

Flows or transformations that expand and contract complimentary directions exponentially are called \emph{hyperbolic}. Using the hyperbolic dynamics of $g_t$, one can show that every affine invariant submanifold is a $g_t$-orbit closure. In particular, 

\begin{thm}
Every affine invariant submanifold is a $\G$-orbit closure. 
\end{thm}

\section{The straight line flow}

For much of this section, a good reference is the survey \cite{MT}.

\subsection{Definition and basic properties} 
Fix a unit length vector $v\in \bC$. The straight line flow on a translation surface $(X,\omega)$ sends each point to the point obtained by starting at that point and moving in the direction $v$ at unit speed for time $t$. This gives for each $\bR$ a map $\phi^v_t:(X,\omega)\setminus B_t\to (X,\omega)$, where $B_t$ is the set of ``bad points" whose straight line flow hits a singularity in time at most $t$. $B_t$ consists of a finite union of line segments.

\begin{figure}[h!]
\centering
\includegraphics[scale=0.35]{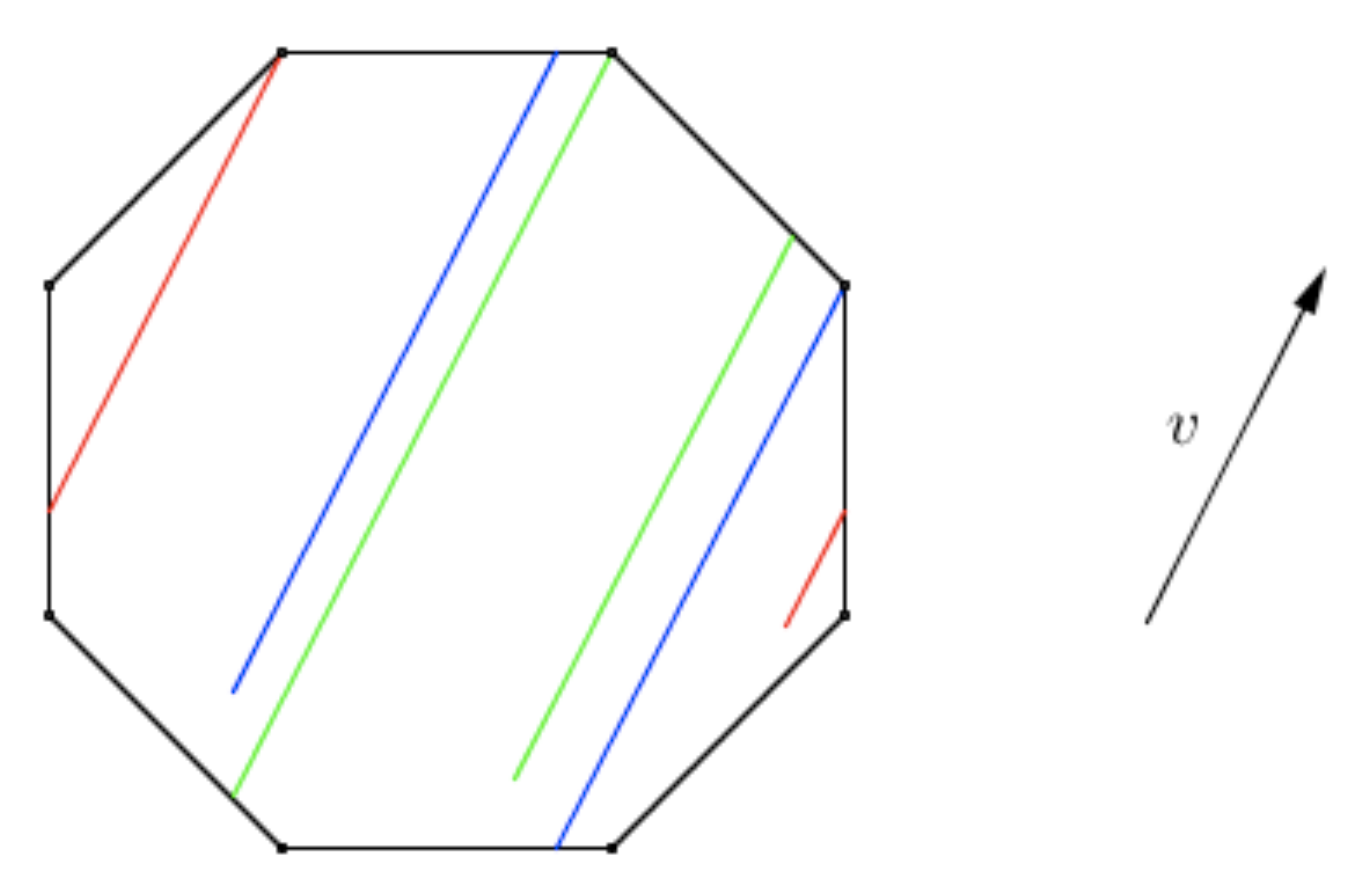}
\caption{The colored line segments consist of (some of the) points whose orbit under straight line flow in direction $v$ hits the singularity in finite time.}
\end{figure}

Let $B=\cup_t B_t$. Although this set might be dense, it has measure zero and hence it should be considered to be of negligible size. The straight line flow is defined on $(X,\omega)\setminus B$ for all time $t$.

One of the reasons straight line flow is important is because billiard trajectories in rational polygons ``unfold" to orbits of straight line flow. This was the original motivating unfolding rational polygons to translation surfaces; instead of bounding off the edges of the polygon, the trajectory can continue straight into a reflected copy of the polygon. 

\begin{figure}[h!]
\centering
\includegraphics[scale=0.35]{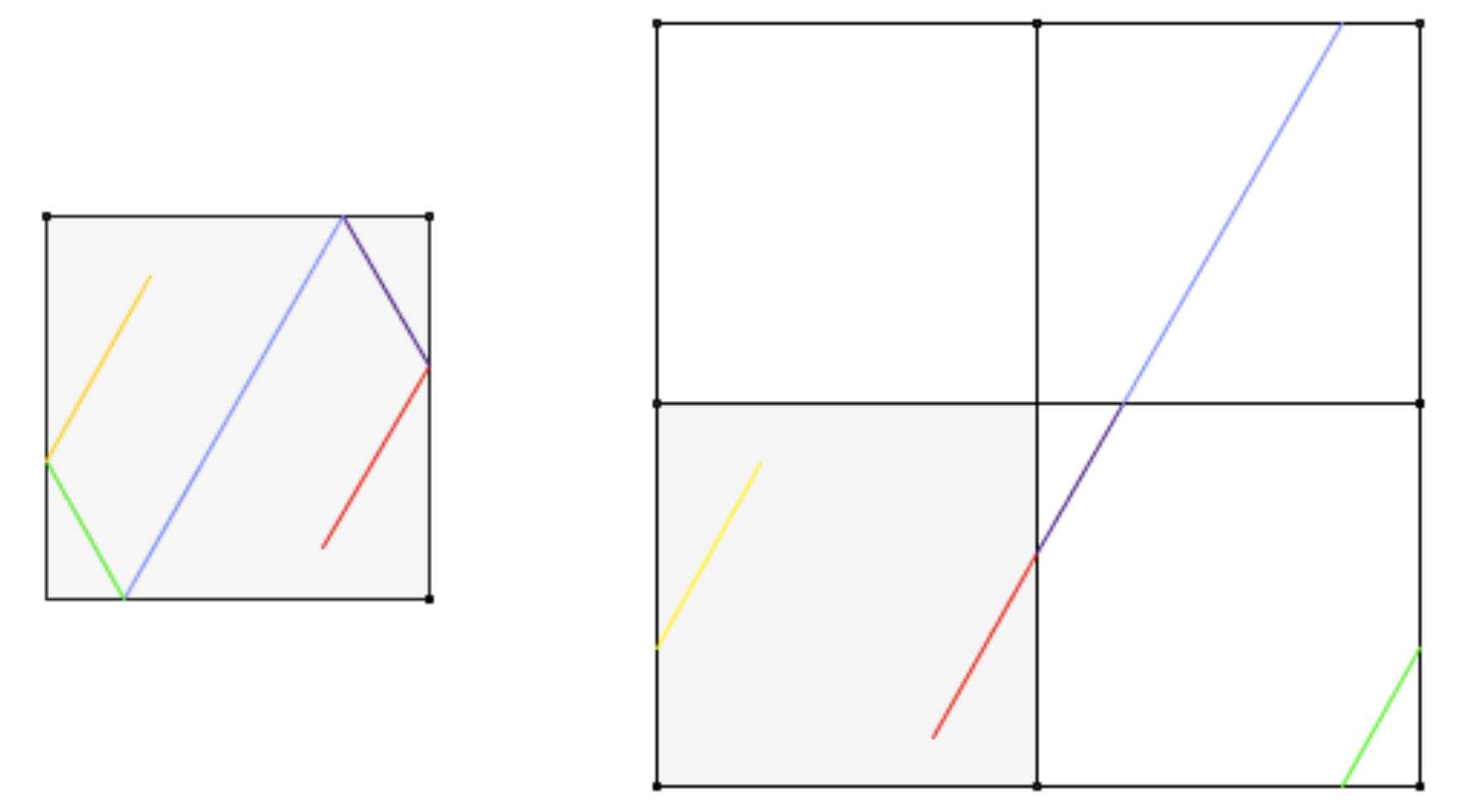}
\caption{A billiard trajectory in the square (left) ``unfolds" to a straight line flow trajectory in the associated translation surface (right).}
\end{figure}

The most basic question about straight line flow is: can the surface be cut into pieces invariant under the straight line flow? The easiest thing to do is to cut out all saddle connections in direction $v$. Possibly this disconnects the surface, possibly not. 

\begin{defn}
A straight line flow is called \emph{minimal} if every orbit that is defined for all time is dense. 
\end{defn}

Let $v$ have irrational slope, and let $(X,\omega)=(\bC/\bZ[i], dz)$. The straight line flow in direction $v$ on $(X,\omega)$ is minimal. The proof of this will follow from a more general result below, but the reader is invited to try to convince themself now that this flow is minimal. 

\begin{figure}[h!]
\centering
\includegraphics[scale=0.42]{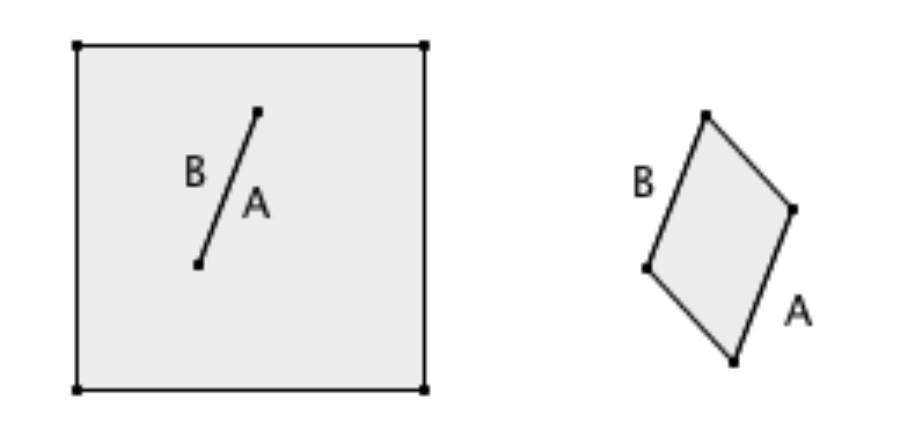}
\caption{Opposite sides are identified, and segment $A$ on the left is identified to segment $A$ on the right, etc., to give a translation surface in $\cH(2)$ that we think of as a cylinder glued onto a slit torus. Then removing the saddle connections $A$ and B disconnects the surface, giving a slit torus (left) and a cylinder (right). If the segment $A$ has irrational slope and the torus is 1 by 1, then the flow in direction $A$ will be minimal in the slit torus, and periodic on the cylinder. }
\end{figure}

\begin{prop}[Katok-Zemljakov \cite{KZ}]\label{P:minimal}
After removing all saddle connections in a given direction on a translation surface $(X,\omega)$, each connected component is either a cylinder or has minimal straight line flow in the given direction. 
\end{prop}

This is an unusual dichotomy: on each component, either every single orbit is periodic (the cylinder case), or every single orbit is dense.

\begin{prop}
On every $(X,\omega)$, the straight line flow is minimal in all but countably many directions. 
\end{prop}

\begin{proof}
There are only countably many saddle connections. In any direction that does not have saddle connections, the previous result says the straight line flow must be minimal. 
\end{proof}

%
%
%

\subsection{Ergodicity} A measure preserving flow on a space $Y$ with a probability measure $m$ is a homomorphism $\phi$ from $\bR$ to the group of invertible measure preserving transformations from $Y$ to itself. (These transformations are typically only considered to be defined on set of full measure.) 

\begin{defn}
A measure preserving flow on a space $Y$ with probability measure $m$ is said to be \emph{ergodic} if $Y$ cannot be written as the disjoint union of two subsets that are invariant under the flow, and each of positive measure.  

The flow is \emph{uniquely ergodic} if $m$ is the unique invariant measure for the flow. (It follows that $m$ is ergodic.) 
\end{defn}

Thus ergodicity is a basic indecomposability condition saying that the dynamics cannot be split into two smaller pieces (each of which could be studied separately). It is somewhat surprising that it implies the following strong restriction on the dynamics, whose proof is nontrivial but can be found in any book on ergodic theory. 

\begin{thm}[Birkhoff Ergodic Theorem]
Let $\phi$ be a measure preserving ergodic flow on a space $Y$ with probability measure $m$, and let $f\in L^1(Y,m)$. Then, for $m$-almost every $y\in Y$, 
$$\lim_{T\to\infty}\frac1{T}\int_0^T \phi_t(y) dt = \int_Y f dm.$$
\end{thm}

You should think of the case where $f=\chi_A$ is the characteristic function of a set $A$. Then the theorem says that long orbit segments $\{\phi_t(y): t\in [0,T]\}$ spend about $m(A)$ of their time in $A$. Thus the Birkhoff Ergodic Theorem says that almost every orbit is \emph{equidistributed}, in that 
$$\lim_{T\to\infty} \frac1T Leb(\{t: t\in [0,T], \phi_t(y)\in A\}) \to m(A),$$
where $Leb$ denotes Lebesgue measure on $\bR$ and $A\subset Y$ is any measurable set. (Technically, to be true as stated, some additional restriction must be place on $A$, for example, $A$ is open. Otherwise any given orbit can be removed from any measurable set $A$, giving a set $A'$ often of the same measure of $A$, but which the given orbit does not intersect at all.) 

%
%
%
%

\bold{Renormalization of straight line flow.} Say that the $g_t$-orbit $$\{g_t(X,\omega), t\geq 0\}$$ is \emph{recurrent} if there is some compact set $K$ of the stratum, such that $g_t(X,\omega)\in K$ for arbitrarily large $t$. This exactly says that $g_t(X,\omega)$ does not diverge to infinity. 

Given an orbit segment of vertical straight line flow on $(X,\omega)$ of length $L$, it yields an orbit segment of vertical straight line flow on $g_t(X,\omega)$ of length $e^{-t}L$. In this way long orbit segments of vertical straight line flow become small under $g_t$, and we say that $g_t$ \emph{renormalizes} the vertical straight line flow. This idea of replacing a long orbit segment of a dynamical segment for a shorter orbit segment of a different but related dynamical system is called renormalization, and is fundamental in dynamics. For the straight line flow, it was used to prove the following. 

\begin{thm}[Masur's criterion \cite{Ma}]
Suppose that $g_t(X,\omega)$ is recurrent. Then the vertical straight line flow on $(X,\omega)$ is uniquely ergodic. 
\end{thm}

The converse is not true, however the result is extremely powerful. 

\begin{thm}[Kerkhoff-Masur-Smillie \cite{KMS}]
For every $(X,\omega)$ and almost every $\theta\in [0,2\pi)$, $g_t(r_\theta(X,\omega))$ is recurrent (as $\theta$ is fixed and $t\to\infty$). Thus, for every $(X,\omega)$, the straight line flow is uniquely ergodic in almost every direction. 
\end{thm}

This implies the same result for billiard flows in rational polygons. 

\begin{rem}
There exist $(X,\omega)$ such that the vertical flow is minimal but not uniquely ergodic. This is a bit strange; every orbit is dense, but most orbits are not equidistributed, and hence somehow favor (spend more time than expected in) some parts of $(X,\omega)$.
\end{rem} 

\subsection{Complete periodicity} The dynamics are much more restricted for $(X,\omega)$ that lie in a 2-dimensional affine invariant submanifold. Such $(X,\omega)$ are called \emph{lattice surfaces}, since their stabilizer is a lattice in $SL(2,\bR)$. 

\begin{thm}[Veech dichotomy \cite{V}]
For any lattice surface, in the direction of any saddle connection the surface is periodic. In all other directions the straight line flow is uniquely ergodic. 
\end{thm}

This is the same dichotomy that holds for $(\bC/\bZ[i], dz)$, where the flow is periodic in the rational directions, and uniquely ergodic in the irrational directions. 

 The proof of the following lemma is easy, but it will be omitted, because it requires familiarity with the geodesic flow on a finite volume complete hyperbolic surface (every geodesic is either recurrent, or it eventually goes straight out a cusp, and the action of $g_t$ on $SL(2,\bR)/\Gamma$ is geodesic flow on the unit tangent bundle to a hyperbolic surface).  

\begin{lem}
For lattice surface $(X,\omega)$, either $g_t (X,\omega)$ is recurrent, or $(X,\omega)$ is stabilized by some matrix 
$$\left(\begin{array}{cc}1&0\\s&1\end{array}\right)$$
with $s\neq 0$. 
\end{lem}

\begin{proof}[Proof of Veech dichotomy]
It suffices to prove the statement for the vertical direction (since the surface can be rotated to make any direction vertical).

By the lemma, either $g_t (X,\omega)$ is recurrent, or $(X,\omega)$ is stabilized by the matrix above. In the first case, Masur's criterion gives that the flow is uniquely ergodic (and it is easy to see there can be no vertical saddle connections, or else $g_t(X,\omega)$ would diverge). In the second case, Proposition \ref{P:para} gives that the surface is the union of vertical cylinders (and so there are vertical saddle connections, on the boundary of the cylinders).
\end{proof}

We will also give a more modern proof of just the first statement, together with a generalization. For this we will need 

\begin{thm}[Minsky-Weiss, Smillie-Weiss \cite{MinW, SW2}]
The $u_t$-orbit closure of any $(X,\omega)$ contains a horizontally periodic surface. 
\end{thm}

Recall $u_t=\left(\begin{array}{cc}1&t\\0&1\end{array}\right)$.

\begin{lem}
For any horizontal cylinder or saddle connection on $(X,\omega)$, and each $t$, there is a corresponding horizontal cylinder or saddle connection on $u_t(X,\omega)$.

 Furthermore, there is a corresponding horizontal cylinder or union of horizontal saddle connections on each translation surface $(X', \omega')$ in the $u_t$-orbit closure of $(X,\omega)$.
\end{lem}

\begin{proof}
The first statement follows because the matrix $u_t$ fixes the horizontal direction. Now consider $(X', \omega')=\lim_{n\to\infty} u_{t_n} (X,\omega)$. For each horizontal cylinder on $(X,\omega)$, there is a corresponding horizontal cylinder on each $u_{t_n} (X,\omega)$, and hence there is also a horizontal cylinder in the limit. 

The same argument applies equally well to horizontal saddle connections, except that possibly in the limit a zero could ``land" on the interior of the saddle connection, subdividing it into several horizontal saddle connections. 
\end{proof}

\begin{prop}[One part of the Veech dichotomy]
Suppose $(X,\omega)$ is a lattice surface, and has a horizontal saddle connection. Then $(X,\omega)$ is horizontally periodic. 
\end{prop}

\begin{proof}
Let $(X', \omega')$ be horizontally periodic and in the $u_t(X,\omega)$-orbit closure of $(X,\omega)$. Let $T$ be large, so $u_T(X,\omega)$ is very close to $(X', \omega')$. The horizontal saddle connection is present on $(X',\omega')$ as a union of horizontal saddle connections (it will turn out to be only one). There must be some matrix $g\in SL(2,\bR)$ close to the identity so $gu_T(X,\omega)=(X', \omega')$, because they are both in the same orbit. Also, $g$ must preserve the horizontal direction, since it must preserve the horizontal saddle connections. However, that means $g$ is a unipotent upper triangular matrix $g=u_S$, so  $u_{T+S}(X,\omega)$ is horizontally periodic, so $(X,\omega)$ is horizontally periodic. 
\end{proof}

\begin{defn}
A \emph{rel deformation} of translation surface is a path $(X_t,\omega_t), t\in [a,b]$, in a stratum, such that for any absolute homology class $\gamma$, 
$\int_\gamma \omega_t$ is constant. 
\end{defn}

Thus along rel deformations absolute periods are constant, but relative periods (i.e., the complex distance between zeros of $\omega$) can change.

\begin{figure}[h!]
\centering
\includegraphics[scale=0.30]{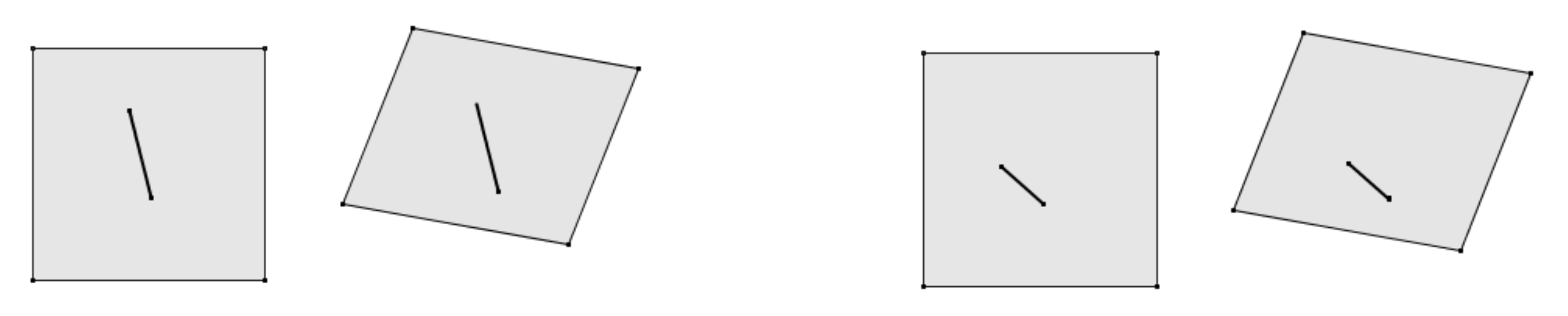}
\caption{The translation surface on the right and the one on the left are rel deformations of each other. More generally, changing the length of the slit in the slit torus construction gives a rel deformation. It is possible to write down a basis for absolute homology consisting of cycles disjoint from the slit, which shows that the integral of any absolute homology class does not change as the complex length of the slit is changed. 
}
\end{figure}

\begin{figure}[h!]
\centering
\includegraphics[scale=0.30]{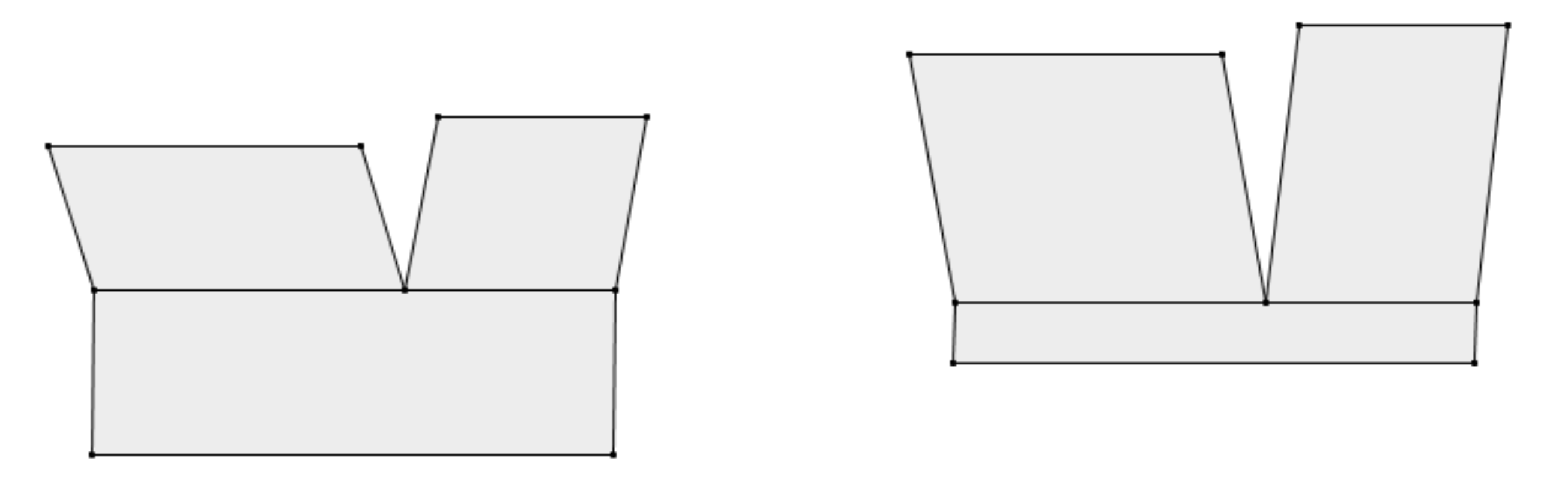}
\caption{Opposite edges are identified (the bottom edges each consist of two saddle connections), giving two surfaces in $\cH(1,1)$. The translation surface on the right and the surface on the left are rel deformations of each other. The surface on the right was obtained from the one on the left by subtracting 1 from the height of the bottom cylinder, and adding 1 to the height of the two top cylinders.}
\label{F:CylsRel}
\end{figure}

\begin{ex}
If $(X',\omega')$ is a translation cover of $(X,\omega)$, then moving the branch points gives a rel deformation of $(X',\omega')$, consisting entirely of surfaces that cover $(X,\omega)$. (Compare to the computations in example \ref{E:cover}.)
\end{ex}

\begin{lem}\label{L:rk1}
Suppose that $(X,\omega)$ and $(X', \omega')$ are nearby surfaces in some stratum, and as subspaces of absolute cohomology $H^1(X,\bC)$ we have
$$\span_\bR(\Re(\omega), \Im(\omega))=\span_\bR(\Re(\omega'), \Im(\omega')).$$
Then there is some $g\in \G$ close to the identity such that $g(X',\omega')$ is a rel deformation of $(X,\omega)$. 
\end{lem}
\begin{proof}
By assumption, there are constants $a,b,c,d\in \bR$ such that, in absolute cohomology,
$$\Re(\omega') = a\Re(\omega)+b\Im(\omega)\quad\text{and}\quad \Im(\omega') = c\Re(\omega)+d\Im(\omega).$$
Thus if $$g=\left(\begin{array}{cc} a&b\\c&d\end{array}\right),$$
then $g(X,\omega)$ and $(X',\omega')$ have the same absolute periods, and hence are rel deformations of each other. In particular, the linear path in local periods joining $g(X,\omega)$ and $(X',\omega')$ is a rel deformation. 
\end{proof}

\begin{defn}
An affine invariant submanifold $\cM$ is \emph{rank 1} if, for every $(X,\omega)\in \cM$, there is an open neighbourhood $U$ containing $(X,\omega)$, such that for every $(X', \omega')\in U$ there is a $g\in \G$ close to the identity such that there is a rel deformation in $U$ from $(X',\omega')$ to $g(X,\omega)$.   
\end{defn}

A definition of rank will be given in the next lecture, so this can be considered a provisional definition of rank 1. It can be rephrased as saying that the $\G$ directions and the directions of rel deformations in $\cM$ span the tangent space to $\cM$ at every point. By the previous lemma, it can also be rephrased as saying that  $\span(\Re(\omega), \Im(\omega))$ is locally constant on $\cM$.

\begin{prop}\label{P:rk1ex}
Let $\cM$ be a 2-dimensional affine invariant submanifold. Then $\cM$ is rank 1. Furthermore, let $\cM'$ be a connected component of the space of degree $d$ translation covers of surfaces in $\cM$ branched over $k$ points. Then $\cM'$ is rank 1 also. 
\end{prop}

\begin{proof}
A 2-dimensional affine invariant submanifold $\cM$ is a single $\G$-orbit, and so is in particular rank 1 (the rel deformations aren't even required). 

If $(X', \omega)\in \cM'$ is a cover of $(X,\omega)\in \cM$, then a neighborhood in $\cM'$ of $(X', \omega')$ is obtained by changing $(X,\omega)$ by a small matrix in $\G$, and changing the location of the branch points. 
\end{proof}

\begin{prop}
The eigenform loci in genus two constructed in the second lecture are rank 1. 
\end{prop}

\begin{proof}
These are defined by saying that the real and imaginary parts of $\omega$ should span a fixed 2-dimensional subspace of absolute homology (the $\sqrt{D}$-eigenspace of $M$).   
\end{proof}

\begin{prop}[Wright]\label{P:CP}
If $(X,\omega)\in \cM$ and $\cM$ is rank 1, then $(X,\omega)$ is periodic in any direction that has a cylinder.
\end{prop}

That is, the proposition says that every time you find a cylinder on $(X,\omega)$, then $(X,\omega)$ is the union of that cylinder and cylinders parallel to it. Before being established in general in \cite{Wcyl} (using a different argument than the one we give here), the proposition was known in several special cases \cite{V, Ca, LN}.

\begin{proof}
 Let $(X', \omega')$ be horizontally periodic and in the $u_t$-orbit closure of $(X,\omega)$.

Let $T$ be large, so $u_T(X,\omega)$ is very close to $(X', \omega')$. The horizontal cylinder is present on $(X',\omega')$. There must be some matrix $g\in SL(2,\bR)$ close to the identity so $gu_T(X,\omega)$ is a rel deformation of $(X', \omega')$.  Small rel deformations preserve cylinders, since the integral over their circumference curve $\gamma$ must remain constant along the rel deformation. Hence $g$ must preserve the horizontal direction, since it must preserve the horizontal cylinder which is present on both $u_T(X,\omega)$ and $gu_T(X,\omega)$.

Thus, every horizontal cylinder on $(X', \omega')$ is also horizontal on $u_T(X,\omega)$. If these cylinders do not cover $u_T(X,\omega)$, then it is possible to derive a contradiction, because then there would be more horizontal cylinders on $(X', \omega')$. (Because every horizontal cylinder on $u_T(X,\omega)$ must also be present on every translation surface in the $u_t$-orbit closure, and if they do not cover $u_T(X,\omega)$ then the corresponding cylinders on $(X', \omega')$ do not cover $(X', \omega')$, and hence there must be more cylinders on $(X', \omega')$ than on $(X,\omega)$.)
\end{proof}

\begin{rem}
In fact this proof shows that if $(X,\omega)$ has a loop $\gamma$ of saddle connections in a fixed direction, and $\int_\gamma \omega\neq 0$, then $(X,\omega)$ is periodic in that direction. 
\end{rem}

We conclude by remarking that there is also a close connection between rank 1 orbit closures and real multiplication. The proof requires dynamics. 

\begin{thm}[Filip]
If $(X,\omega)\in \cM$ and $\cM$ is rank 1, then $\Jac(X,\omega)$ has real multiplication. 
\end{thm}

The conclusion is that rank 1 orbit closures are very close cousins to 2-dimensional orbit closures, which are a special case. 

\section{Revisiting genus two with new tools}

\subsection{Field of definition, VHS}

\begin{defn} The \emph{(affine) field of definition} $\bk(\cM)$ of an affine invariant submanifold $\cM$ is the smallest subfield of $\bR$ such that $\cM$ can be defined in local period coordinates by linear equations with coefficients in this field \cite{Wfield}. Warning: This is not the same thing as the field of definition of $\cM$ viewed as a variety (where polynomial equations are allowed, and the coordinates are different).
\end{defn}  

For example, $\cM$ arising from branched covers over tori, (or over all surfaces in some other stratum) are defined over $\bQ$. The eigenform loci are defined over $\bQ[\sqrt{D}]$. 

 Let $H^1$ denote the flat bundle over $\cM$ whose fiber over $(X, \omega)\in \cM$ is $H^1(X,\bC)$, and let $H^1_{rel}$ denote the flat bundle whose fiber over $(X,\omega)$ is $H^1(X,\Sigma, \bC)$, where $\Sigma$ is the set of singularities of $(X,\omega)$. Let $$p:H^1_{rel}\to H^1$$ denote the natural projection from relative to absolute cohomology. Viewing $$H^1(X,\Sigma, \bC) = H_1(X,\Sigma, \bC)^*\quad\text{and}\quad H^1(X, \bC)=H_1(X, \bC)^*,$$ the map $p$ is just restriction of a linear functional on $H_1(X,\Sigma, \bC)$ to $H_1(X, \bC)\subset H_1(X,\Sigma, \bC)$. The subspace $\ker(p)$ exactly corresponds to derivatives of rel deformations. 
 
 The flat connection on $H^1$ or $H^1_{rel}$ is often called the Gauss-Manin connection. From our point of view, it is an extremely simple thing. The cohomology groups $H^1(X,\bC)$ and $H^1(X, \Sigma, \bC)$ are both purely topological objects, and do not depend on the complex structure on $X$ or the Abelian differential $\omega$. (When the Abelian differential changes, the set $\Sigma$ might move a bit by an isotopy.) Thus, varying the complex structure on $X$ does not change these cohomology groups. In this way, if $(X', \omega')$ is nearby $(X,\omega)$, then $H^1(X, \bC)$ is identified with $H^1(X', \omega')$ (because it is the same exact object!), and similarly for relative cohomology. This identification of nearby fibers is exactly the structure of a flat connection. 
 
Recall that period coordinates can be considered as the map sending $(X,\omega)$ to the relative cohomology class of $\omega$ in $H^1(X,\Sigma, \bC)$. By definition, $\cM$ is defined in these periods by a linear subspace, which we think of as simultaneously giving $\cM$ in period coordinates, as well as being the tangent space to $\cM$ at $(X,\omega)$. (The tangent space to a vector space, at any point, is just the vector space itself.) Thus we can consider the tangent bundle $T(\cM)$  of $\cM$ as a flat subbundle of $H^1_{rel}$. It is flat because the subspace defining $\cM$ in period coordinates does not change as $(X,\omega)$ moves around in $\cM$. 

A \emph{flat subbundle} $E$ of $H^1$ or $H^1_{rel}$ is just a subbundle that is locally constant over $\cM$. Associated to such a subbundle is its \emph{monodromy representation}, which is a representation of $\pi_1(\cM)$ on a fiber of $E$ at the chosen base point of $\cM$. It is obtained by dragging cohomology classes around loops in $\pi_1(\cM)$. 

A flat subbundle is called \emph{simple} if it has no nontrivial flat subbundle, or equivalently if the monodromy representation has no nontrivial invariant subspaces. A flat subbundle is called \emph{semisimple} if it is the direct sum of simple subbundles. Two subbundles are called Galois conjugate if their fibers are Galois conjugate; in particular, this means their monodromy representations are Galois conjugate. 

\begin{defn} The \emph{field of definition} of a flat subbundle $E\subset H^1$ is the smallest subfield of $\bC$ such that locally the linear subspace $E$ of $H^1(X,\bC)$ can be defined by linear equations (with respect to an integer basis of $H_1(X, \bZ)$) with coefficients in this field.  The trace field of a flat bundle over $\cM$ is defined as the field generated by traces of the corresponding representation of $\pi_1(\cM)$.
\end{defn}

\begin{thm}[Wright \cite{Wfield}]\label{T:galois}
Let $\cM$ be an affine invariant submanifold. The field of definition of $p(T(\cM))$ and trace field of $p(T(\cM))$ are both equal to $\bk(\cM)$. 

Set $\bV_{\Id}=p(\cT(\cM))$. There is a semisimple flat bundle $\bW$, and for each field embedding $\rho:\bk(\cM)\to\bC$ there is a flat simple bundle $\bV_\rho$ that is Galois conjugate to $\bV_{\Id}$, such that
\[H^1 = \left(\bigoplus_\rho \bV_\rho\right) \oplus \bW.\]

The bundle $\bW$ does not contain any subbundles isomorphic to any $\bV_\rho$. Both $\bW$ and $\oplus \bV_\rho$ are defined over $\bQ$. All direct summands are symplectic and symplectically orthogonal. 

In particular, 
\[\dim_\bC p(T(\cM)) \cdot \deg_\bQ \bk(\cM) \leq 2g.\]
\end{thm}

\begin{cor}[Wright]
In particular, the field of definition is a number field, and so any translation surface whose coordinates are linearly independent over $\overline{\bQ}$ cannot be contained in a nontrivial affine invariant submanifold, and hence must have  $\G$-orbit closure as large as possible. This provides an explicit full measure set of surfaces whose orbit closure is as large as possible. 
\end{cor}

The direct sum decomposition of $H^1$ in Theorem \ref{T:galois} was previously established in the case of Teichm\"uller curves by Martin M\"oller \cite{M}, and is one of the main tools used in the study of closed $SL(2,\bR)$--orbits. 

\begin{thm}[M\"oller \cite{M}, Filip \cite{Fi1}]
The splitting above is a splitting of Variation of Hodge Structures. That is, each direct summand is equal to the direct sum of its intersection with $H^{1,0}$ and its intersection with $H^{0,1}$. 
\end{thm}

Again the proofs use dynamics. For an elementary introduction of Variation of Hodge Structures in the context of orbit closures, see \cite{W1}. 

When the splitting of $H^1$ is nontrivial ($H^1\neq p(T(\cM))$), then $\cM$ parameterizes translation surfaces whose Jacobians admit nontrivial endomorphisms.

Given that affine invariant submanifolds are varieties, parts (but not all) of the two theorems above follow from a theorem of Deligne on semi-simplicity of VHS \cite{D1}. However, in fact both the above theorems were established first, and used by Filip in his proof that affine invariant submanifolds are varieties. The first theorem used work of Avila-Eskin-M\"oller \cite{AEM}, and the following.

\begin{thm}[Eskin-Mirzakhani-Rafi \cite{EMR}]
In every affine invariant submanifolds, there are lots of closed (i.e., periodic) $g_t$-orbits. 
\end{thm}

This is useful because of a classical result in Teichm\"uller theory that says the monodromy matrix over such an orbit has simple largest and smallest eigenvalues, with eigenvectors $\Re(\omega)$ and $\Im(\omega)$. Among other things, this helps to show that there is not a second copy of $p(T(\cM))$ in the decomposition of $H^1$, since otherwise these eigenvalues would have multiplicity at least two. 

The work of Avila-Eskin-M\"oller used in the proof also  shows

\begin{thm}[Avila-Eskin-M\"oller]
$p(T(\cM))$ is symplectic. 
\end{thm}

\begin{defn}
The \emph{rank} of an affine invariant submanifold $\cM$ is $\frac12 \dim_\bC p(T(\cM))$. 
\end{defn}

When $p(T(\cM))$ is 2-dimensional, that means that it must be spanned by the real and imaginary parts of the absolute cohomology classes given by $\omega$. In particular, $\span(\Re(\omega), \Im(\omega))$ is locally constant, since $p(T(\cM))$ is always a flat subbundle. By Lemma \ref{L:rk1}, this means that a neighbourhood of any $(X,\omega)$ in $\cM$ can be generated using $\G$ and rel deformations, so rank 1 according to this definition agrees with our previous definition of rank 1. 

\subsection{Cylinder deformations} We begin with the observation that, for each cylinder on a translation surface $(X,\omega)$, there is a corresponding cylinder on sufficiently nearby surfaces $(X', \omega')$. This corresponding cylinder may not have the same direction, height, circumference, or modulus, however these all change continuously with $(X', \omega')$. The ``sufficiently nearby" assumption is required since, along a path a surfaces starting at $(X,\omega)$, the height of the cylinder might reach 0, at which point the cylinder ceases to be a cylinder. 

\begin{figure}[h!]
\centering
\includegraphics[scale=0.22]{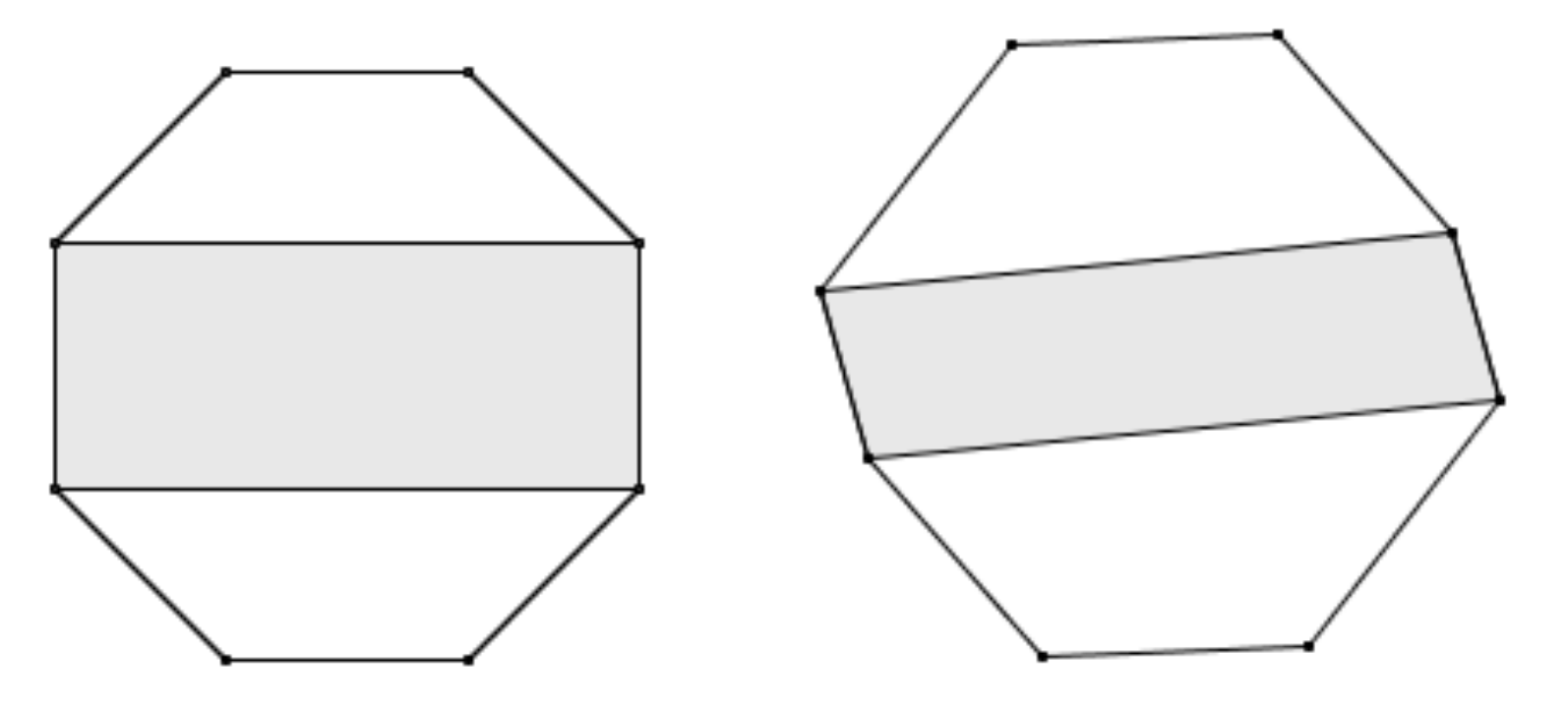}
\caption{Consider the shaded horizontal cylinder on the regular octagon surface in $\cH(2)$ (left). On any sufficiently small deformation of this surface, the cylinder persists (right).}
\label{F:Persist}
\end{figure}

\begin{defn}Two cylinders $C_1$ and $C_2$ on $(X,\omega) \in \cM$ are said to be \emph{$\cM$-parallel} if they are parallel, and remain parallel on all deformations of $(X,\omega)$ in $\cM$. (A \emph{deformation} is just a nearby surface, connected to $(X,\omega)$ via a path in $\cM$.) The deformations are assumed to be small, so that $C_1$ and $C_2$ persist. 
\end{defn}

\begin{ex}
If $\cM$ is 2-dimensional, it consists of a single $\G$-orbit. Parallel cylinders remain parallel under the action of $\G$, so two cylinders on a translation surface in $\cM$ are $\cM$-parallel if and only if they are parallel. 
\end{ex}

\begin{ex}
On the opposite extreme, if $\cM$ is a connected component of a stratum, two cylinders are $\cM$ parallel if and only if they are parallel and their circumference curves are homologous. Indeed, suppose the core curves are $\gamma_1$ and $\gamma_2$. Then since $\int_{\gamma_1} \omega'=\int_{\gamma_2} \omega'$ on all deformations $(X', \omega')\in \cM$, the cylinders will always be in the same direction, which is exactly the direction in $\bC$ given by this integral. 
\end{ex}

\begin{figure}[h!]
\centering
\includegraphics[scale=0.30]{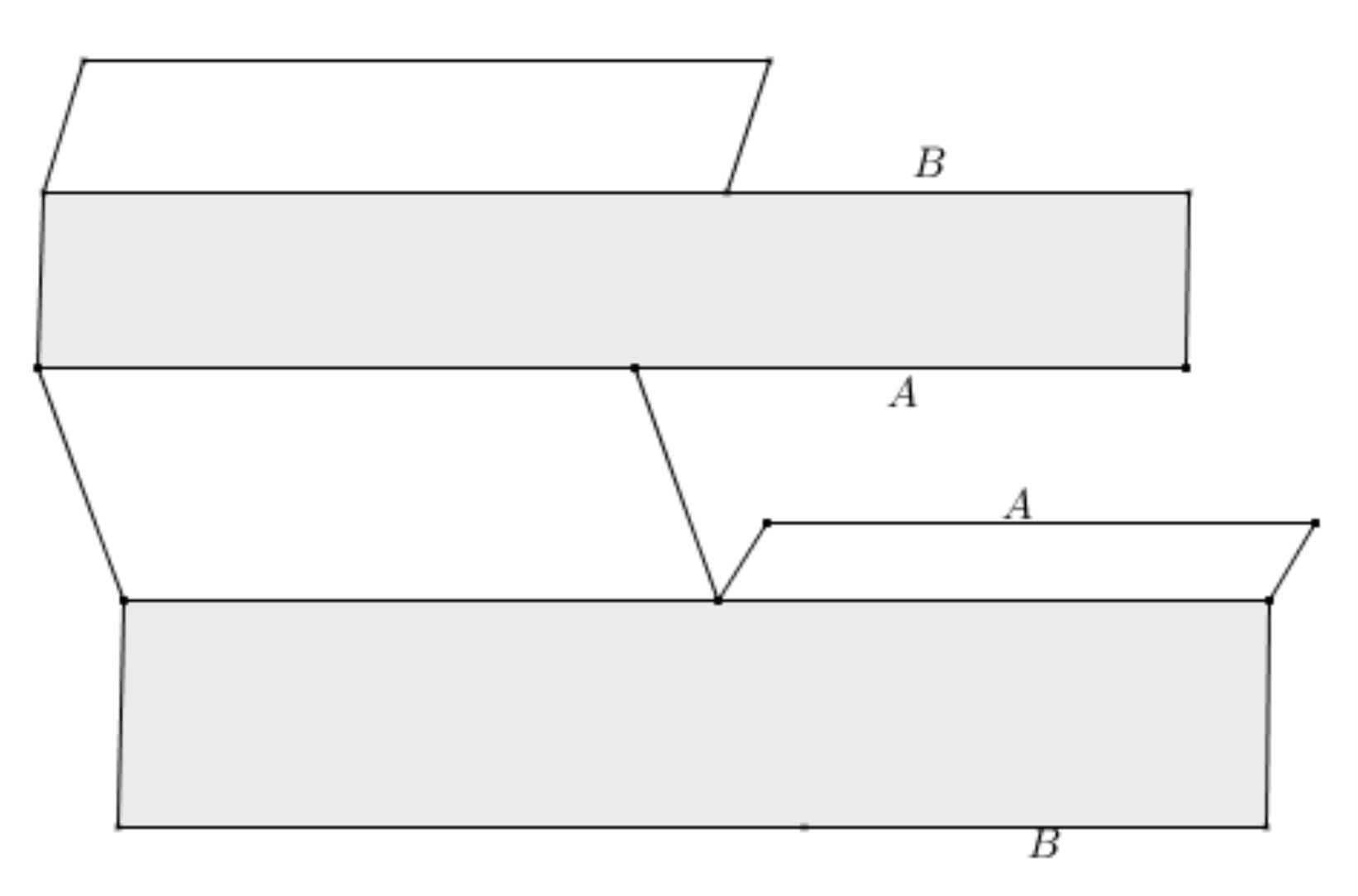}
\caption{The two shaded cylinders are homologous on this translation surface. There is no way to deform the surface to make these two cylinders not be parallel.}
\end{figure}

\begin{lem}\label{L:parallel}
Two cylinders $C_1$ and $C_2$ on $(X,\omega)\in \cM$ are $\cM$-parallel if and only if there is a constant $c\in \bR$ such that on all deformations $(X', \omega')\in\cM$, 
$$\int_{\gamma_1} \omega'=c\int_{\gamma_2} \omega',$$
where $\gamma_i$ is the circumference curve of $C_i$. 
\end{lem}

In other words, two cylinders on $(X,\omega)\in \cM$ are $\cM$-parallel if and only if one of the linear equations defining $\cM$ in local period coordinates makes it so.

\begin{proof}
Both $(X', \omega')\mapsto \int_{\gamma_1}\omega'$ and $(X', \omega')\mapsto\int_{\gamma_2}\omega'$ are linear functionals on a neighbourhood of $(X,\omega)$ in $\cM$, viewed in period coordinates as locally being an open set in a complex vector space. The cylinders $C_1, C_2$ are $\cM$ parallel if and only if these two linear functions are always real multiples of each other. 

The only way for two linear functionals on a complex vector space always have real ratio is for one functional to be a fixed real constant $c$ times the other. 
\end{proof}

The relation of being $\cM$-parallel is an equivalence relation, and when we speak of an equivalence class of a cylinder, we mean the set of all cylinders $\cM$-parallel to it.

Define the matrices
$$u_t = \left(\begin{array}{cc} 1 & t\\0 & 1 \end{array}\right), \quad a_s = \left(\begin{array}{cc} 1 & 0\\0 & e^s \end{array}\right),\quad r_\theta = \left(\begin{array}{cc} \cos \theta  & -\sin \theta  \\\sin \theta  & \cos \theta  \end{array}\right).$$
Let $\cC$ be a collection of parallel cylinders on a translation surface $(X,\omega)$. Suppose they are all of angle $\theta \in[0,\pi)$. Define $a_s^\cC(u_t^\cC(X,\omega)))$ to be the translation surface obtained by applying $r_{-\theta}$ to $(X,\omega)$, then applying the matrix $a_s u_t$ to the images of the cylinders in $\cC$, and then applying $r_\theta$. 

The result of $a_s^\cC$ is to stretch the height of all cylinders in the collection $\cC$ by a factor of $e^s$. The result of $u_t^\cC$ is to shear all the cylinders in $\cC$. 

\begin{figure}[h!]
\centering
\includegraphics[scale=0.30]{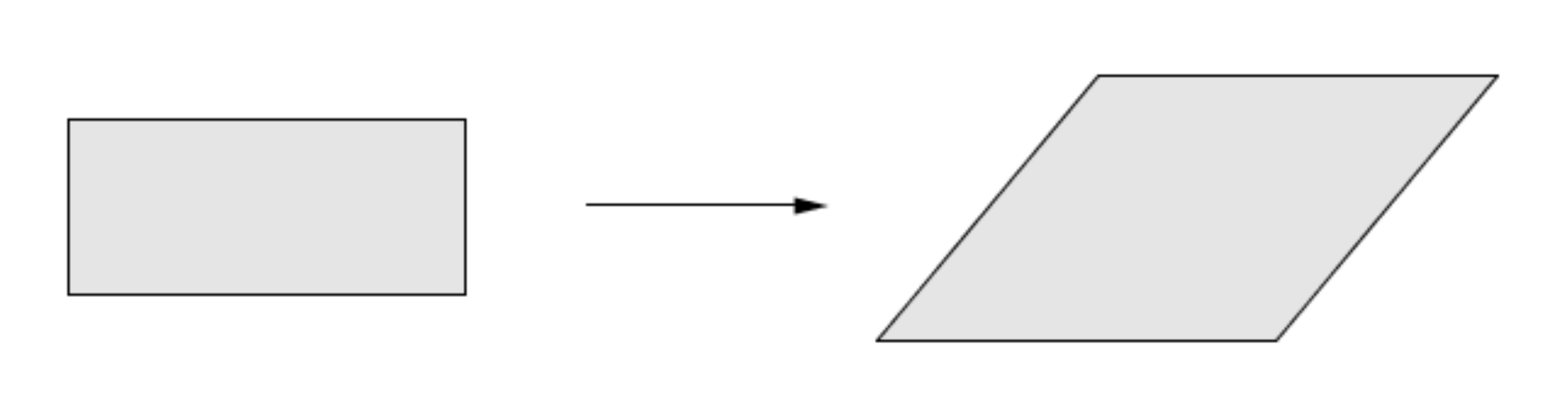}
\caption{A horizontal cylinder (left--vertical sides are identified, but the top and bottom horizontal edges are the boundary of the cylinder.) The result of stretching and shearing this cylinder (right). Note the boundary of the cylinder stays exactly the same.}
\end{figure}

There is more than one way of thinking about these cylinder deformations. You can think of cutting out the cylinders in $\cC$, and then stretching and shearing them, and then gluing them back in. Or, you can think of a polygon decomposition for $(X,\omega)$, with one parallelogram for each cylinder in $\cC$, and you can think of applying the matrix $a_su_t$ just to the parallelograms giving cylinders in $\cC$ and doing nothing to the remaining polygons. 

\begin{figure}[h!]
\centering
\includegraphics[scale=0.28]{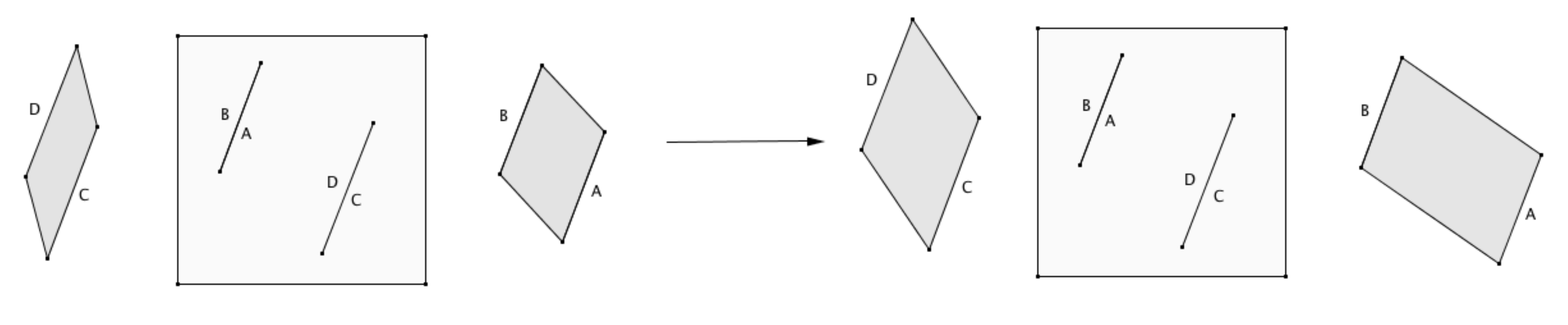}
\caption{On the left is a translation surface $(X,\omega)$. Let $\cC$ be the two shaded parallel cylinders on $(X,\omega)$. On the right is $a_{\log(2)}^\cC(X,\omega)$. }
\end{figure}

\begin{thm}[The Cylinder Deformation Theorem \cite{Wcyl}]\label{T:CDT}
Suppose that $\cC$ is an equivalence class of $\cM$-parallel cylinders on $(X,\omega)\in \cM$. Then for all $s, t\in \bR$, the translation surface $a_s^\cC(u_t^\cC(X,\omega))\in \cM$. 
\end{thm}

We call $a_s^\cC$ the \emph{cylinder stretch}, and $u_t^\cC$ the \emph{cylinder shear}.

The proof involves the dynamics of $u_t$, but is motivated by the then conjecture (now theorem) that affine invariant submanifolds are varieties. As $s\to\infty$, the translation surface $a_s^\cC(\cC(X,\omega))$ converges to the boundary at infinity of $\cM$. The Cylinder Deformation Theorem is closely related to the boundary structure of $\cM$.

More generally, the boundary of an affine invariant submanifold is related to configurations of parallel cylinders and saddle connections, since these can always be made vertical using $r_\theta$ and then shrunk using $g_t$. 

 \subsection{Orbit closures in genus 2} In this subsection, we will give a qualitative description of orbit closures in genus 2. However, we will not give a classification, in that we will not  discuss how many connected components the loci we discuss have, and we will not discuss how to tell if two translation surfaces have the same orbit closure. An almost complete classification in genus two was given by McMullen, the only remaining open problem being to classify orbits of square-tiled surfaces in $\cH(1,1)$. (Such a classification should give a finite list of invariants, such that if two square-tiled surfaces have the same invariants, then they are in the same orbit.) Thus all results in this section can be deduced as particular consequences of  finer results of McMullen. 

See  \cite{McM:spin, Mc4, Mc6, Ba, Mu:orb, Mc5, KM} for much finer information on orbit closures in genus 2. Note that McMullen's work \cite{Mc5} was done well before the work of Eskin-Mirzakhani-Mohammadi, and hence uses different techniques than what we present here. 

\begin{lem}
In $\cH(2)$ every $\G$-orbit is dense or closed. Every closed orbit is either the orbit of a square-tiled surface, or one of the eigenform loci in $\cH(2)$ constructed in lecture 2. 
\end{lem}

Even the orbits of square-tiled surfaces can be fit into the framework of eigenforms by considering ``real multiplication by $\bQ\oplus \bQ$", but we do not pursue that perspective here. 

\begin{proof}
Every orbit closure is an affine invariant submanifold. In strata with only one zero, $p:H^1(X,\Sigma, \bC)\to H^1(X,\bC)$ is an isomorphism, and hence $T(\cM)=p(T(\cM)$ is symplectic. Hence because $\cM\subset \cH(2)$, we see that  $\cM$ is 2-dimensional or 4-dimensional. $\cH(2)$ is 4-dimensional and connected, so any $4$-dimensional affine invariant submanifold is all of $\cH(2)$. 

Thus we may assume that $\cM$ is 2-dimensional, i.e., a closed orbit. If $\cM$ is defined over $\bQ$, it contains surface $(X,\omega)$ whose periods $$\left\{\int_\gamma \omega: \gamma\in H_1(X,\bZ)\right\} \subset \frac1n \bZ[i]$$ for some $n$. This implies $(X,\omega)$ is square-tiled surface (a cover of $\bC/(\frac1n \bZ[i])$). 

If $\cM$ is defined over $\bQ[\sqrt{D}]$, then it follows from Theorem \ref{T:galois} that 
$$\span(\Re(\omega), \Im(\omega))$$ is defined over $\bQ[\sqrt{D}]$ and is symplectically orthogonal to the Galois conjugate subspace. Hence by lecture 2, $(X,\omega)$ is an eigenform for real multiplication by an order in $\bQ[\sqrt{D}]$. 
\end{proof}

Similarly we obtain

\begin{lem}
In $\cH(1,1)$, every 3 dimensional rank 1 orbit closure consists either of torus covers, or of eigenforms for real multiplication by $\bQ[\sqrt{D}]$. 
\end{lem}

The only other dimension a rank 1 orbit closure could have in $\cH(1,1)$ is 2, and there are indeed closed orbits in $\cH(1,1)$. McMullen has shown that all but one contains a square-tiled surface \cite{Mc4}. This one exceptional closed orbit is a two dimensional submanifold of the locus of eigenforms for real multiplication for $\bQ[\sqrt{5}]$.

To complete the qualitative classification of orbit closures in $\cH(1,1)$, it remains only to show the following. 

\begin{prop}
Any affine invariant submanifold of $\cH(1,1)$ that is not rank 1 must be all of $\cH(1,1)$. 
\end{prop}

\begin{proof}
Since $p(T(\cM))$ is symplectic, and rank is $\frac12 \dim p(T(\cM))$, we see that $p(T(\cM))=H^1$. (In genus 2, $H^1$ is 4-dimensional.) We wish to show that in fact $\cM$ is 5-dimensional, in which case it must be all of $\cH(1,1)$. 

\begin{lem}\label{L:3cyls}
Such $\cM$ must contain a horizontally periodic surface with $3$ horizontal cylinders. Two of the cylinders each have a single saddle connection in their boundary, and both of these are glued to a third larger cylinder above and below, as in figure \ref{F:AllFree}. 
\end{lem}

\begin{figure}[h!]
\centering
\includegraphics[scale=0.28]{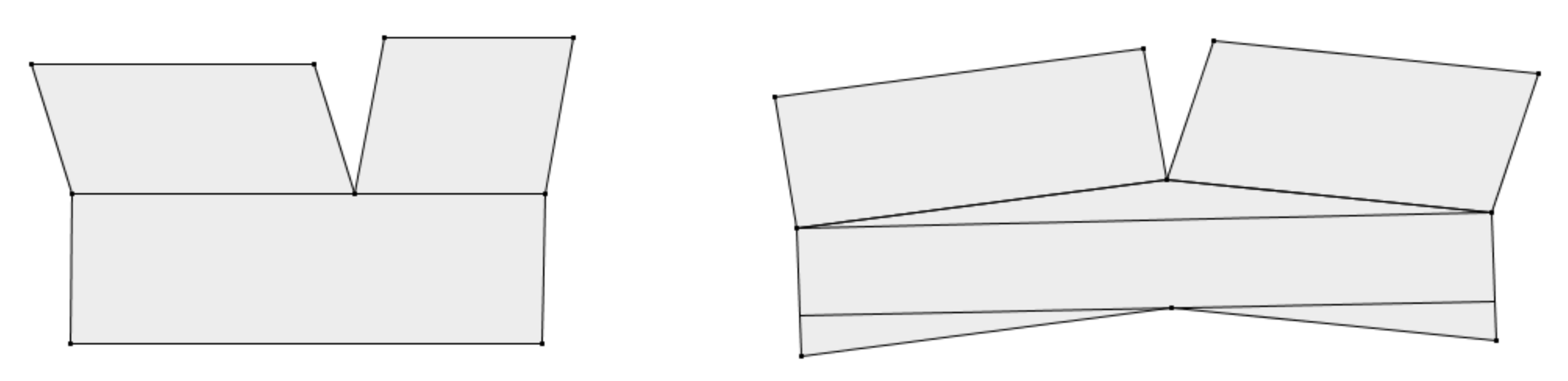}
\caption{On the left is a translation surface in $\cH(1,1)$. On the right is an illustration of a deformation of this surface where the three horizontal cylinders cease to be parallel. (This exact deformation is not guaranteed to be in $\cM$, however one like it is.)}
\label{F:AllFree}
\end{figure}

Given the lemma, the proposition is not so hard to prove. Indeed, it is easily verified that none of the three cylinders are $\cM$ parallel to each other: By Lemma \ref{L:parallel}, two cylinders are $\cM$ parallel if and only if the integrals of their circumference curves are always proportional. However, since $p(T(\cM))=H^1(X, \bC)$, no restriction is placed on the absolute periods of a translation surface in $\cM$, and we can change the circumferences in a arbitrary way. (As we do so, the relative period might be determined by the absolute ones--but this is precisely what we will show cannot happen, since $\cM$ will end up being $\cH(1,1)$.)

Since none of the three cylinders are $\cM$ parallel to each other, we can increase the height of top two, and decrease the height of the bottom one, so as to produce a rel deformation in $\cM$, as in figure \ref{F:CylsRel}. Since $p(T(\cM))=H^1(X,\cM)$, and $T(\cM)$ contains a vector in the 1-dimensional $\ker(p)$ (the rel deformation), we see that $T(\cM)=H^1(X,\Sigma, \bC)$. Hence $\cM$ is 5-dimensional, and $\cM=\cH(1,1)$. 
\end{proof}

The proof of Lemma \ref{L:3cyls} is easy but will be omitted. (One can find a horizontally periodic surface in every horocycle flow orbit closure, and work from there to get one with three horizontal cylinders. See figure \ref{F:Getting3}, and compare to arguments in \cite{NW, ANW}.) 

\begin{figure}[h!]
\centering
\includegraphics[scale=0.3]{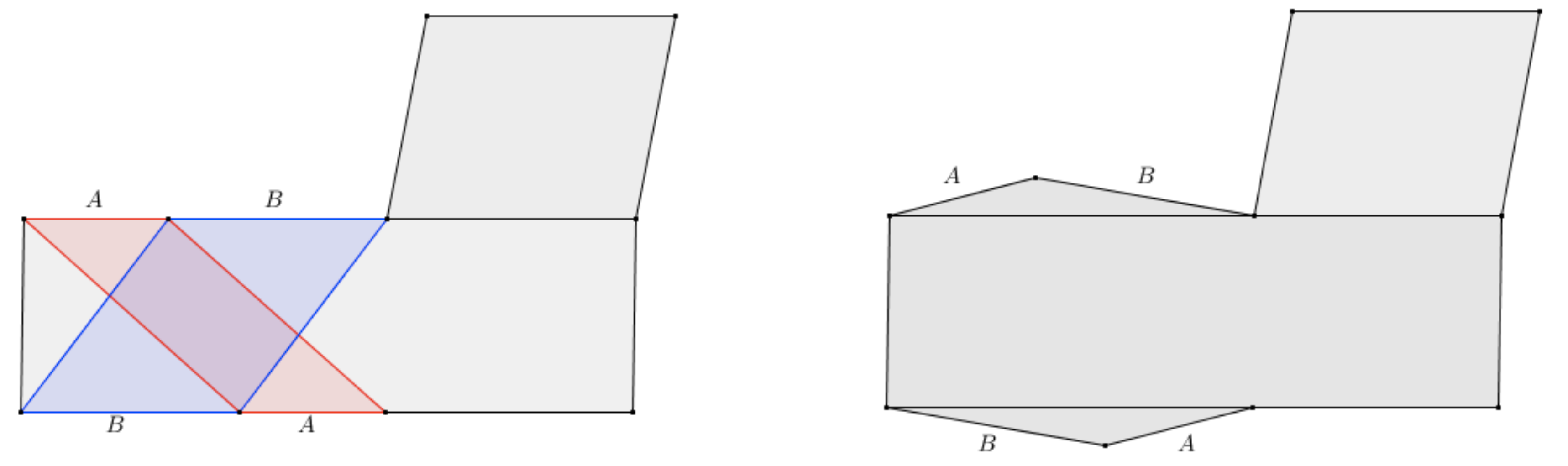}
\caption{The proof of Lemma \ref{L:3cyls} is omitted, but this figure gives the idea. On the left is a translation surface with two horizontal cylinders. Shearing the two shaded nonhorizontal cylinders appropriately ``opposite" amounts gives the surface on the right, which now has three horizontal cylinders as desired.}
\label{F:Getting3}
\end{figure}

\subsection{Census of known orbit closures} Here we give a list of the currently known orbit closures not arising from the elementary constructions discussed in Section 2.1. We give only primitive orbit closures, i.e., those not arising from covering constructions. 

\bold{Closed orbits.} McMullen and Calta independently discovered infinitely many closed orbits in $\cH(2)$ \cite{Ca, Mc}. McMullen generalized his approach using real multiplication to find infinitely many primitive closed orbits in genus 3 and 4, in the Prym loci in the strata where $\omega$ has only one zero \cite{Mc2}.   

There is a  bi-infinite sequence of Teichm\"uller curves $\cT(n,m)$ called the Veech-Ward-Bouw-M\"oller curves. For $n=2$ they were discovered by Veech \cite{V}; for $n=3$ by Ward \cite{W}; and in the general case by Bouw and M\"oller \cite{BM}. They have also been studied by the author and Hooper \cite{W2, H}.  

There are also two ``sporadic" closed orbits known, one due to Vorobets in $H(6)$, and another due to Kenyon-Smillie in $H(1,3)$ \cite{HS, KS}. These sporadic examples correspond to billiards in the  $(\pi/5, \pi/3, 7\pi/15)$ and $(2\pi/9, \pi/3, 4\pi/9)$ triangles respectively. 

\bold{Rank 1 affine invariant manifolds.} McMullen and Calta discovered infinitely many rank 1 affine invariant submanifolds in $\cH(1,1)$ \cite{Ca, Mc}. McMullen generalized his construction to give infinitely many additional examples in the Prym loci in genus 3, 4, and 5 \cite{Mc2}. 

\bold{A new orbit closure.} In joint work in progress with Mirzakhani, the author has found the first known example of a higher rank orbit closure whose affine field of definition is not $\bQ.$ 




\bibliography{mybib}{}
\bibliographystyle{amsalpha}

\end{document}